\documentclass[a4paper,12pt]{amsart}

\usepackage[utf8]{inputenc}
\usepackage[T1]{fontenc}
\usepackage[american]{babel}

\usepackage[margin=7em]{geometry}

\usepackage{amssymb, amsfonts}

\usepackage{bbm}

\usepackage{enumerate}

\numberwithin{equation}{section}

\newtheorem{theoremcounter}{theoremcounter}[section]

\newtheorem{algorithm}[theoremcounter]{Algorithm}
\newtheorem{conjecture}[theoremcounter]{Conjecture}
\newtheorem{corollary}[theoremcounter]{Corollary}
\newtheorem{definition}[theoremcounter]{Definition}

\newtheorem{lemma}[theoremcounter]{Lemma}
\newtheorem{proposition}[theoremcounter]{Proposition}
\newtheorem{remark}[theoremcounter]{Remark}
\newtheorem{remarks}[theoremcounter]{Remarks}
\newtheorem{theorem}[theoremcounter]{Theorem}

\renewcommand{\frak}{\ensuremath{\mathfrak}}
\newcommand{\cal}{\ensuremath{\mathcal}}
\newcommand{\bboard}{\ensuremath{\mathbb}}

\newcommand{\frake}{\ensuremath{\frak{e}}}

\newcommand{\cD}{\ensuremath{\cal{D}}}
\newcommand{\cE}{\ensuremath{\cal{E}}}
\newcommand{\cF}{\ensuremath{\cal{F}}}

\newcommand{\cL}{\ensuremath{\cal{L}}}

\newcommand{\bbP}{\ensuremath{\bboard P}}

\newcommand{\rmF}{\ensuremath{\mathrm{F}}}
\newcommand{\rmG}{\ensuremath{\mathrm{G}}}
\newcommand{\rmH}{\ensuremath{\mathrm{H}}}
\newcommand{\rmJ}{\ensuremath{\mathrm{J}}}

\newcommand{\rmM}{\ensuremath{\mathrm{M}}}

\newcommand{\rmT}{\ensuremath{\mathrm{T}}}

\newcommand{\rmg}{\ensuremath{\mathrm{g}}}

\newcommand{\rmr}{\ensuremath{\mathrm{r}}}

\newcommand{\amid}{\ensuremath{\mathop{\mid}}}

\newcommand{\ZZ}{\ensuremath{\mathbb{Z}}}
\newcommand{\QQ}{\ensuremath{\mathbb{Q}}}
\newcommand{\RR}{\ensuremath{\mathbb{R}}}
\newcommand{\CC}{\ensuremath{\mathbb{C}}}

\renewcommand{\Im}{\ensuremath{\mathop{\mathfrak{Im}}}}

\newcommand{\isdiv}{\amid}
\newcommand{\nisdiv}{\ensuremath{\mathop{\nmid}}}

\renewcommand{\pmod}[1]{\ensuremath{\;(\mathrm{mod}\, #1)}}

\newcommand{\Mat}[1]{\ensuremath{\mathrm{Mat}_{#1}}}
\newcommand{\MatT}[1]{\ensuremath{\mathrm{Mat}^\rmT_{#1}}}

\newcommand{\GL}[1]{\ensuremath{\mathrm{GL}_{#1}}}
\newcommand{\SL}[1]{\ensuremath{\mathrm{SL}_{#1}}}

\newcommand{\Sp}[1]{\ensuremath{\mathrm{Sp}_{#1}}}
\newcommand{\Orth}[1]{\ensuremath{\mathrm{O}_{#1}}}

\newcommand{\T}{\ensuremath{\mathrm{T}}}

\newcommand{\tr}{\ensuremath{\mathrm{tr}}}

\newcommand{\rk}{\ensuremath{\mathop{\mathrm{rk}}}}

\newcommand{\slashdiv}{\ensuremath{\mathop{/}}}

\newcommand{\lspan}{\ensuremath{\mathop{\mathrm{span}}}}

\newcommand{\HS}{\mathbb{H}}

\newcommand{\td}{\tilde}

\newcommand{\ud}{\underline}

\newcommand{\tGL}[1]{\ensuremath{\widetilde{\rm GL}_{#1}}}
\newcommand{\Aff}[1]{\ensuremath{{\rm Aff}_{#1}}}
\newcommand{\tAff}[1]{\ensuremath{\widetilde{\rm Aff}_{#1}}}
\newcommand{\ord}{\ensuremath{{\rm ord}}}
\newcommand{\tSp}[1]{\ensuremath{\widetilde{\rm Sp}_{#1}}}
\newcommand{\tGamma}{\ensuremath{\widetilde{\Gamma}}}

\newcommand{\CH}{\ensuremath{{\rm CH}}}

\newcommand{\rot}{\ensuremath{{\rm rot}}}

\newcommand{\disc}{\ensuremath{{\rm disc}}}

\renewcommand{\tr}{\ensuremath{\mathrm{trace}}}

\renewcommand{\div}{\ensuremath{\mathrm{div}}}

\begin{document}

\title[Formal Fourier Jacobi expansions]{Formal Fourier Jacobi expansions \\ and special cycles of codimension~$2$}

\author{Martin Raum}
\address{ETH, Dept. Mathematics, Rämistraße 101, CH-8092, Zürich, Switzerland}
\email{martin.raum@math.ethz.ch}
\urladdr{http://www.raum-brothers.eu/martin/}
\thanks{The author is supported by the ETH Zurich Postdoctoral Fellowship Program and by the Marie Curie Actions for People COFUND Program.}


\begin{abstract}
We prove that formal Fourier Jacobi expansions of degree~$2$ are Siegel modular forms.  As a corollary, we deduce modularity of the generating function of special cycles of codimension~$2$, which were defined by Kudla.  A second application is the proof of termination of an algorithm to compute Fourier expansions of arbitrary Siegel modular forms of degree~$2$.  Combining both results enables us to determine relations of special cycles in the second Chow group.
\end{abstract}

\maketitle

\section{Introduction and statement of results}
\label{sec:introduction}

There is a close interplay between Siegel modular forms and Jacobi forms, which, for example, was employed to prove the Saito-Kurokawa conjecture~\cite{An79, Ma79a, Ma79b, Ma79c, Za81}.  This connection is founded on the Fourier Jacobi expansion of Siegel modular forms, which can be turned into a formal notion.  Formal Fourier Jacobi expansions of degree~$2$ and weight~$k \in \ZZ$, in the simplest case, are formal series with respect to $q'$ of the form
\begin{gather*}
  \sum_{0 \le m \in \ZZ} \phi_m \, q^{\prime \, m}
\end{gather*}
such that:
\begin{enumerate}[(i)]
\item The $\phi_m$ are Jacobi forms of weight~$k$ and index~$m$ (cf.~\cite{EZ85} and Section~\ref{sec:jacobi-forms} for a definition).
\item The Fourier coefficients of $\phi_m$ satisfy $c(\phi_m; n, r) = c(\phi_n; m, r)$.
\end{enumerate}
A more precise definition will be given later.

Ibukiyama, Poor, and Yuen studied such expansions in~\cite{IPY12}, and at the end of the introduction they asked the question whether every Fourier Jacobi expansion is convergent.  Theorem~\ref{thm:main-theorem-fourier-jacobi} answers this question in the affirmative.
\begin{theorem}
Every formal Fourier Jacobi expansions of degree~$2$ and of any weight converges, and hence is the Fourier Jacobi expansion of a Siegel modular form.
\end{theorem}
\noindent A more detailed statement is presented in Theorem~\ref{thm:main-theorem-fourier-jacobi}.

The significance of formal Fourier Jacobi expansions stems from their relation to families of certain algebraic cycles on Shimura varieties of orthogonal type.  The study of cycles, that is, subvarieties of varieties, is a prominent theme in algebraic geometry, pursued, e.g., in~\cite{Le86, Fa90, Zh97}.  Codimension~$r$ cycles on a variety~$X$ are classified up to rational equivalence by the $r$th Chow group $\CH^r(X)_\CC$, first studied in~\cite{Ch56}.  In most cases, little is now about it.  Using definite subspaces in indefinite quadratic spaces, Kudla produced cycles of arbitrary codimension on Shimura varieties of orthogonal type, which he called ``special cycles''~\cite{Ku97}.  Special cycles generalize interesting geometric configurations.  For examples, CM points on modular curves and Hirzebruch's and Zagier's curves $T_N$ on Hilbert modular surfaces~\cite{HZ76} are included in this notion.  As special cycles are related explicitly to geometry of lattices, one hopes, they are easier to study than cycles that do not come with this additional information.  Kudla conjectured the following.
\begin{conjecture}
\label{conj:kudla}
The generating function of special cycles of codimension~$r$ is a Siegel modular form of degree~$r$.
\end{conjecture}
\begin{theorem}[{Borcherds~\cite{Bo99,Bo00}}]
Conjecture~\ref{conj:kudla} is true for $r = 1$.
\end{theorem}
As our main application, we establish a further case.
\begin{theorem}
\label{thm:conj:kudla-is-true}
Conjecture~\ref{conj:kudla} is true for $r = 2$.
\end{theorem}
\noindent A more specific description of this modularity result is contained in Theorem~\ref{thm:main-theorem-kudla}.

Our affirmation of Kudla's prediction in the case $r = 2$ yields rich geometric information.  It allows us to compute relations that were inaccessible before.  Compare this with results obtain in the 80's by Kudla and Millson.  They examined modularity of generating functions of intersection numbers of special cycles, which, they found, arise from (additive) theta lifts~\cite{KM86, KM87, KM90}.  This way, they discovered relations of such intersection numbers.


\subsection{Siegel modular forms}
\label{ssec:introduction:siegel-modular-forms}

We now discuss the theory of formal Fourier Jacobi series in greater detail.  Afterward, we turn our attention to implications for geometry.

In order to explain Theorem~\ref{thm:main-theorem-fourier-jacobi}, which implies Theorem~\ref{thm:main-theorem-kudla} and hence Theorem~\ref{thm:conj:kudla-is-true}, we need to make several definitions.  For the time being, we restrict ourselves to the case of integral weights.  Fix $k, l \in \ZZ$, and let $\rho$ be a finite dimensional, unitary representation of the symplectic group $\Sp{2}(\ZZ) \subset \Mat{4}(\ZZ)$ with representation space~$V_\rho$.  Denote the $l$th symmetric power of the canonical representation of $\GL{2}(\CC)$ on $\CC^2$ by $\sigma_l$, and write $V_l$ for its representation space.  A (doubly vector valued) Siegel modular forms is a holomorphic function~$\Phi$ from the Siegel upper half space~$\HS^{(2)}$ to $V_l \otimes V_\rho$ that is invariant under the modular transformations
\begin{gather*}
  (\det{}^{k} \otimes \sigma_l)^{-1}(C Z + D )\, \rho(\gamma)^{-1}\, \Phi\big((A Z + B) (C Z + D)^{-1}\big)
\quad
 \big( \gamma = \left(\begin{smallmatrix}A & B \\ C & D\end{smallmatrix}\right) \in \Sp{2}(\ZZ) \big)
\text{.}
\end{gather*}
Every Siegel modular form~$\Phi$ admits an absolutely convergent Fourier Jacobi expansion
\begin{gather*}
  \sum_{0 \le m \in \QQ} \phi_m(\tau, z) \, q^{\prime\, m}
\text{,}
\end{gather*}
where $\phi_m \in \rmJ_{k, m}(\sigma_l \otimes \rho)$ is a (doubly vector valued) Jacobi form~(cf.~\cite{EZ85, Sk08, IK11} and Section~\ref{sec:jacobi-forms}).  Jacobi forms have likewise Fourier expansions
\begin{gather*}
  \phi_m(\tau, z)
=
  \sum_{0 \le n \in \QQ,\, r \in \QQ} c(\phi_m; n, r)\, q^n \zeta^r
\text{.}
\end{gather*}
By modularity of~$\Phi$, their Fourier coefficients satisfy
\begin{gather}
\label{eq:involution-condition-introduction}
  c(\phi_m; n, r)
=
  (\det{}^k \otimes \sigma_l)^{-1}(S)\, \rho^{-1}(\rot(S))\, c(\phi_n; m, r)
\end{gather}
for all $0 \le m, n \in \QQ$ and all $r \in \QQ$, where
\begin{gather*}
  S
=
  \left(\begin{smallmatrix} 0 & 1 \\ 1 & 0 \end{smallmatrix}\right)
\in
  \GL{2}(\ZZ)
\quad\text{and}\quad
  \rot(S)
=
  \left(\begin{smallmatrix} S & \\ & S \end{smallmatrix}\right)
\in
  \Sp{2}(\ZZ)
\text{.}
\end{gather*}
Our main theorem can be considered as a converse.
\begin{theorem}
\label{thm:main-theorem-fourier-jacobi}
Let
\begin{gather*}
  \sum_{0 \le m \in \QQ} \phi_m(\tau, z) \, q^{\prime\, m}
\end{gather*}
be a formal expansion in $q'$ with coefficients in $\rmJ_{k, m}(\sigma_l \otimes \rho)$, where $k \in \frac{1}{2} \ZZ$ and $l \in \ZZ$.  If the Fourier coefficients of all $\phi_m$ satisfy (\ref{eq:involution-condition-introduction}) for all $0 \le m, n \in \QQ$ and all $r \in \QQ$, then this series converges absolutely and defines a Siegel modular form.
\end{theorem}
The proof of Theorem~\ref{thm:main-theorem-fourier-jacobi}, which is an unrefined version of Theorem~\ref{thm:formal-fourier-jacobi-series-equal-siegel-modular-forms}, relies on comparing dimensions and on a regularity result for meromorphic functions in two variables.  In the case of even weight~$k$, we bound the asymptotic dimension of spaces of formal Fourier Jacobi expansions as $k \rightarrow \infty$.  It equals the asymptotic dimension of corresponding spaces of Siegel modular forms.  In the proof of Theorem~\ref{thm:formal-fourier-jacobi-series-equal-siegel-modular-forms}, we fix a formal Fourier Jacobi expansion and multiply it with a space of classical Siegel modular forms of sufficiently large weight.  By a dimension argument, we find that the resulting space of formal Fourier Jacobi expansions (of then higher weight) contains at least one Siegel modular form.  From this, we are able to deduce that the formal Fourier Jacobi expansions which we started with represents a meromorphic Siegel modular form.  We then apply our regularity result~(Theorem~\ref{thm:siegel-regularity}) to conclude that it is a holomorphic Siegel modular form. For $k \in \frac{1}{2} \ZZ$, we obtain an even weight Siegel modular form by multiplication (or tensoring) with another suitable Siegel modular form.  This allows us to reduce the general case to the case of even weights.

The special case of $l = 0$ and trivial $\rho$ was considered in~\cite{IPY12}.  Ibukiyama, Poor, and Yuen showed straightforwardly that dimensions of the space of formal Fourier Jacobi expansions and the space of Siegel modular forms are equal.  Their technique, however, relies on very detailed knowledge of dimension formulas, which is not available in the more general setting.

\subsection{Special cycles}

Zhang proved that the generating functions of special cycles, as presented later, are formal Fourier Jacobi expansions~\cite{Zh09}.  He pulled back along embeddings of Shimura curves and employed Borcherds's previous result.  For codimension~$2$ cycles, Theorem~\ref{thm:main-theorem-fourier-jacobi} implies modularity.  In this manner, our approach shows that Kudla's modularity conjecture can be examined by combining a detailed analysis of the codimension~$1$ case -- already pursued by Borcherds -- and rigidity results for formal expansions of Siegel modular forms.

We give details on special cycles and their generating functions.  Let $\cL$ be a lattice of signature $(n, 2)$; Write $\cL^\#$ for its dual.  We consider Shirmura varieties~$X_\Gamma$ of orthogonal type associated to subgroups~$\Gamma$ of $\Orth{}(\cL)$ that act trivially on $\disc\,\cL = \cL^\# \slashdiv \cL$.  By definition, we have
\begin{gather*}
  X_\Gamma
=
  \Gamma \backslash \rmG\rmr^-(\cL \otimes \RR)
\text{,}
\end{gather*}
where $\rmG\rmr^-$ denotes the Grassmannian of maximal negative definite subspaces.  Note that for our choice of $\cL$, $X_\Gamma$ carries the structure of a complex, quasi projective variety.

Given a positive semi-definite, symmetric ``moment matrix'' $0 \le T \in \MatT{r}(\QQ)$ and $\mu \in \disc^r\, \cL = (\disc\,\cL)^r$, there is a special cycle $\{ Z(T, \mu) \} \in \CH^\bullet(X_\Gamma)_\CC$ whose codimension equals the rank of~$T$, $\rk(T)$.  In Section~\ref{sec:special-cycles}, we give a precise definition.
Write $\omega^\vee$ for the anti-canonical bundle on $X_\Gamma$, and denote the canonical basis elements of $\CC[ \disc^r\,\cL ]$ by $\frake_\mu$ ($\mu \in \disc^r\, \cL$).  We prove the following:
\begin{theorem}
\label{thm:main-theorem-kudla}
Let $\Gamma \subset \Orth{}(\cL)$ be a subgroup that fixes $\disc\, \cL$.  Then
\begin{gather*}
  \sum_{\substack{ \mu \in \disc^2\, \cL \\ 0 \le T \in \MatT{2}(\QQ)}}
  \{Z(T, \mu)\} \cdot \{\omega^\vee \}^{2 - \rk(T)}\,
  \exp\big(2 \pi i\, \tr( T Z ) \big)\;
  \frake_\mu
\end{gather*}
is a vector valued Siegel modular form with coefficients in $\CH^2(X_\Gamma)_\CC \otimes \CC[ \disc^2\, \cL ]$, which has weight~$1 + \frac{n}{2}$.
\end{theorem}
\noindent As explained in Section~\ref{ssec:introduction:siegel-modular-forms}, there are two notions of vector valued Siegel modular forms.  In Theorem~\ref{thm:main-theorem-kudla} we refer to the one that comes from representations of $\Sp{2}(\ZZ)$ (or its double cover).  Note also that $X_\Gamma$ is not compact.  Its Chow ring can be viewed as the quotient $\CH^\bullet(Y_\Gamma)_\CC \slashdiv \CH^\bullet (\partial Y_\Gamma)$, where $Y_\Gamma$ denotes the toroidal compactification of~$X_\Gamma$ and~$\partial Y_\Gamma$ are the boundary components.

As an outcome of Theorem~\ref{thm:main-theorem-kudla}, we are able to prove the following statement, which was shown by Zhang using adhoc methods~\cite{Zh09}.
\begin{corollary}
The span of special cycles $\{Z(T, \mu)\}$ in $\CH^2(X_\Gamma)_\CC$ is finite dimensional.
\end{corollary}

\subsection{Computing relations of special cycles}

Let $\rho_{\cL, 2}$ denote the type of the generating function in Theorem~\ref{thm:main-theorem-kudla}.  Its representation space is the group algebra $\CC\big[ \disc^2\, \cL \big]$.  Denote by $\Phi_\mu$ ($\mu \in \disc^2\, \cL$) the components of a vector valued Siegel modular form $\Phi$ of type $\rho_{\cL, 2}$, and write $c(\Phi_\mu; T)$ for its $T$th Fourier coefficient.
\begin{corollary}
\label{thm:main-corollary-relations}
Suppose that $b :\, \MatT{2}(\QQ) \times \disc^2\, \cL \rightarrow \CC$ is a function with finite support such that 
\begin{gather*}
  \sum_{T, \mu} b(T, \mu)\, c(\Phi_\mu; T)
=
  0
\end{gather*}
for every Siegel modular form of weight $1 + \frac{n}{2}$ and type $\rho_{\cL, 2}$.  Then we have
\begin{gather*}
  \sum_{T, \mu} b(T, \mu)\, \{ Z(T, \mu) \}
=
  0
\in
  \CH^2( X_\Gamma )_\CC
\text{.}
\end{gather*}
\end{corollary}
Explicit knowledge about Fourier coefficients of Siegel modular forms thus yields results about relations in $\CH^2( X_\Gamma )_\CC$.  This observation motivates our considerations in Section~\ref{sec:computing-via-fj-expansions}.  However, we warn the reader that relations computed in this way do not necessarily exhaust all relations that hold for the $\{ Z(T, \mu) \}$, even thought this is conceivable, if $\cL$ satisfies some mild hypotheses.  The computational investigation of relations for $r = 1$, was resolved by the author in~\cite{Ra12}, building on work of Bruinier.  According to~\cite{Br12}, if $r = 1$ and $\cL$ splits off a hyperbolic plane and a unimodular hyperbolic plane, all relations that hold for the $\{ Z(T, \mu) \}$ ($0 \le T \in \QQ$, $\mu \in \disc\, \cL$) can be described by the analog of Corollary~\ref{thm:main-corollary-relations}.

We prove correctness and termination of an algorithm to compute Fourier expansions of Siegel modular forms of degree~$2$.  It is very much inspired, and in a sense that we make clear in a moment, generalizes ideas conveyed in~\cite{Po10, IPY12}.
\begin{theorem}
There is an explicit, terminating algorithm that computes Fourier expansions of Siegel modular forms of arbitrary weight and type.
\end{theorem}
The proof of this finding is very similar to the proof of Theorem~\ref{thm:main-theorem-fourier-jacobi}.  We employ variants for truncated Fourier expansions of lemmas that entered the latter proof.

Combining the above and Corollary~\ref{thm:main-corollary-relations}, we obtain
\begin{corollary}
\label{cor:main-corollary-computing-relations}
There is an explicit, terminating algorithm that computes infinitely many relations of special cycles $\{ Z(T, \mu) \} \in \CH^2(X_\Gamma)_\CC$.
\end{corollary}

In~\cite{Po10, IPY12}, Ibukiyama, Poor, and Yuen -- following a suggestion by Armand Brumer -- hinted an algorithm to compute paramodular forms, which is very similar to ours.  But they were not able to prove that theirs terminates, except in four cases.  Since paramodular forms are Siegel modular forms for congruence subgroups of Hecke type, our algorithm also applies to their problem after inducing the trivial representation of this subgroup to $\Sp{2}(\ZZ)$.  In this sense, we resolve the question raised at the end of the introduction of~\cite{IPY12}, asking whether ``the growth condition is superfluous''.  In practice, however, our algorithm will most certainly be slower than Poor and Yuen's, and proving that it terminates without employing induction of representations would be useful.

\subsection{Final remarks}

In Theorem~\ref{thm:main-theorem-fourier-jacobi}, we have also treated the case of vector valued weight $l \ne 0$, which we have not yet made use of.  We have included it into our considerations, because of work by Bergstr\"om, Farber, and van der Geer~\cite{FvG04, FvG04b, BFvG08}.  They described a method that computes quiet efficiently Hecke eigenvalues of Siegel modular forms for all $k$ and $l$ if $\rho$ is trivial.  These computation rely on extensive precomputations, that would have to be redone if one wanted to study Siegel modular forms for congruence subgroups other than those that they examined.  Our method then provides an alternative path to follow, which even yields further information.
\vspace*{1ex}

The paper is organized as follows.  In Section~\ref{sec:notation}, we briefly fix notation.  Section~\ref{sec:siegel-modular-forms} and Section~\ref{sec:jacobi-forms} contain general discussions of Siegel modular forms and Jacobi forms.  The proof of Theorem~\ref{thm:main-theorem-fourier-jacobi} is contained in Section~\ref{sec:formal-fj-expansions}.  Our main application is then discussed in Section~\ref{sec:special-cycles}.  We consider computations of Fourier expansion in the final Section~\ref{sec:computing-via-fj-expansions}.

{\it Acknowledgment: To be entered after the referee's report is received.}

\section{Notation}
\label{sec:notation}

We write $I_n$ for the $n \times n$ identity matrix.  The zero matrix of size $n \times m$ is denoted by $0^{(n, m)}$.  If $m = 1$, we suppress the corresponding superscript.  The module of $n \times n$ matrices with entries in a ring~$R$ is denoted by $\Mat{n}(R)$.  We write $\cdot\,{}^\T$ for transposition of matrices.  The space of symmetric matrices of size~$n$ is denoted by $\MatT{n}(R)$.

The group of affine transformations on $\CC^n$ is denoted by $\Aff{n}(\CC) \cong \GL{n}(\CC) \ltimes \CC^n$.  It can be embedded into $\GL{n + 1}(\CC)$ via
\begin{gather}
\label{eq:inclusion-aff-gl}
  (\gamma, \mu)
\mapsto
  \begin{pmatrix}
   \gamma & \mu \\
   0^{(1, n)} & 1
  \end{pmatrix}
\text{.}
\end{gather}
Restrictions of representations of $\GL{2}(\CC)$ to $\Aff{1}(\CC)$ that appear in this paper are taken with respect to this embedding.  When referring to representations, we always mean finite dimensional representations.  The scalar representation of $\GL{n}(\CC)$ that lets $g$ act as $\det(g)^k$ is denoted by~$\det{}^k$.  The $l$th symmetric power of the canonical representation of $\GL{n}(\CC)$ ($n \ge 2$) on $\CC^n$ is denoted by $\sigma_l$.

Both $\GL{n}(\CC)$ and $\Aff{n}(\CC)$ have connected double covers $\tGL{n}(\CC)$ and $\tAff{n}(\CC)$.  We think of elements in $\GL{n}(\CC)$ as pairs $(A, \det(A)^{\frac{1}{2}})$, where the second component is one choice of root.  The one dimensional representation of $\tGL{n}(\CC)$ which lets $(A, \det(A)^{\frac{1}{2}})$ act as $\det^{\frac{1}{2}}(A)$ is denoted by $\det{}^{\frac{1}{2}}$.  We have $\tAff{n}(\CC) \cong \tGL{n}(\CC) \ltimes \CC^n$, and we will use this isomorphism without further mentioning it.  The above embedding of $\Aff{n}(\CC)$ into $\GL{n + 1}(\CC)$ extends to an embedding of $\tAff{n}(\CC)$ into $\tGL{n + 1}(\CC)$.

Representations of discrete groups are assumed to have finite index kernels.  This applied in particular to representations of $\tGL{n}(\ZZ)$ and $\tSp{g}(\ZZ)$.

\section{Siegel modular forms}
\label{sec:siegel-modular-forms}

Let $J_g = \left(\begin{smallmatrix} 0^{(g,g)} & -I_g \\ I_g & 0^{(g,g)} \end{smallmatrix}\right)$.
The full Siegel modular group of degree~$g$ is
\begin{gather*}
  \Sp{g}(\ZZ)
=
  \big\{ \gamma \in \Mat{2g}(\ZZ) \,:\, \gamma^\T J_g \gamma = J_g \big\}
\end{gather*}
with typical elements $\gamma = \left(\begin{smallmatrix} A & B \\ C & D \end{smallmatrix}\right)$ for $A, B, C, D \in \Mat{g}(\ZZ)$.  One kind of elements that we use frequently is
\begin{gather*}
  \rot(U)
=
  \left(\begin{smallmatrix}
    U & \\
    & (U^{-1})^\T
  \end{smallmatrix}\right)
\end{gather*}
for $U \in \GL{g}(\ZZ)$.  The connected double cover $\tSp{g}(\ZZ)$ of $\Sp{g}(\ZZ)$ is the non trivial central extension of $\Sp{g}(\CC)$ by the cyclic group $C_2$.

The full Siegel modular group acts on the Siegel upper half space of degree~$g$
\begin{gather*}
  \HS^{(g)}
=
  \big\{ Z = X + i Y \in \Mat{g}(\CC) \,:\, Y > 0 \big\}
\end{gather*}
via
\begin{gather*}
  \gamma Z
=
  (A Z + B) (C Z + D)^{-1}
\text{.}
\end{gather*}
We obtain an action of $\tSp{g}(\ZZ)$ on $\HS^{(g)}$ by composing with the projection $\tSp{g}(\ZZ) \twoheadrightarrow \Sp{g}(\ZZ)$.  In this paper, we will focus on the case $g \le 2$, and so we fix special notation in both cases.  We usually write $\tau$ for elements of $\HS^{(1)}$, and $\left(\begin{smallmatrix} \tau & z \\ z & \tau' \end{smallmatrix}\right)$ for elements of $\HS^{(2)}$.  Further, we set $q = \exp( 2 \pi i\, \tau)$, $\zeta = \exp( 2 \pi i\, z)$, and $q' = \exp(2 \pi i\, \tau')$.

We can use the upper half space to find a concrete description of elements in $\tSp{g}(\ZZ)$.  Elements of $\tSp{g}(\ZZ)$ can and will be thought of as pairs
\begin{gather*}
  \big( \gamma,\, \big((C Z + D), \det(C Z + D)^\frac{1}{2} \big) \big)
\text{.}
\end{gather*}
The first component is an element of $\Sp{g}(\ZZ)$, and the second component is a holomorphic map $\HS^{(g)} \rightarrow \tGL{g}(\CC)$, that we typically denote by $\omega$.

Let $\sigma$ be a representation of $\tGL{g}(\CC)$, the double cover of $\GL{g}(\CC)$, with representation space $V_\sigma$.  Let $\rho$ be a representation of $\tSp{g}(\ZZ)$ with representation space $V_\rho$.  The Siegel slash action  of weight~$\sigma$ and type~$\rho$ is defined as
\begin{gather*}
  \big( \Phi \big|_{\sigma, \rho}\, (\gamma, \omega) \big) (Z)
=
  \sigma(\omega(Z))^{-1}  \rho(\gamma)^{-1} \, \Phi(\gamma Z)
\end{gather*}
for all $\Phi :\, \HS^{(g)} \rightarrow V_{\sigma} \otimes V_\rho$ and all $(\gamma, \omega) \in \tSp{g}(\ZZ)$.

\begin{definition}
\label{def:siegel-modular-forms}
Let $\sigma$ be a representation of $\tGL{g}(\CC)$ with representation space $V_\sigma$, and let $\rho$ be a representation of $\tSp{g}(\ZZ)$ with representation space $V_\rho$.  A function $\Phi :\, \HS^{(g)} \rightarrow V_{\sigma} \otimes V_\rho$ is called a Siegel modular form of weight $\sigma$ and type $\rho$ if $\Phi \big|_{\sigma, \rho}\, (\gamma, \omega) = \Phi$ for all $(\gamma, \omega) \in \tSp{g}(\ZZ)$ and, in addition, if $g = 1$, $\Phi$ has Fourier expansion
\begin{gather*}
  \Phi(\tau)
=
  \sum_{0 \le n \in \QQ} c(\Phi; n) \, q^n
\text{,}\quad
  c(\Phi; n) \in V_\sigma \otimes V_\rho
\text{.}
\end{gather*}

If $\Phi$ is meromorphic and satisfies $\Phi \big|_{\sigma, \rho}\, (\gamma, \omega) = \Phi$ for all $(\gamma, \omega) \in \tSp{g}(\ZZ)$, we call it a meromorphic Siegel modular form.
\end{definition}
\noindent We write $\rmM^{(g)}(\sigma, \rho)$ for the space of Siegel modular forms of weight~$\sigma$ and type~$\rho$.

Irreducible representations of $\tGL{g}(\CC)$ either factor through $\GL{g}(\CC)$, or they are of the from $\det{}^{\frac{1}{2}} \otimes \sigma'$, where $\sigma'$ factors through $\GL{g}(\CC)$.  The meaning of $\det{}^k$ with $k \in \frac{1}{2}\ZZ$ thus becomes clear.  The space of Siegel modular forms of degree~$1$ (i.e., elliptic modular forms) of weight~$\det^k$, $k \in \frac{1}{2} \ZZ$ and type $\rho$ is denoted by $\rmM^{(1)}_k(\rho)$.  The space degree~$2$ Siegel modular forms of weight~$\det^k \otimes \sigma_l$, $k \in \frac{1}{2} \ZZ$, $l \in \ZZ$ and type~$\rho$ is denoted by  $\rmM^{(2)}_{k, l}(\rho)$.  If $l = 0$, we suppress it, and if $\rho$ is trivial, we abbreviate $\rmM^{(1)}_k(\rho) = \rmM^{(1)}_k$ and $\rmM^{(2)}_{k, l}(\rho) = \rmM^{(2)}_{k, l}$.

\begin{remark}
Suppose that $\sigma$ and $\rho$ are irreducible.  If $\sigma$ factors through $\GL{g}(\CC)$, one can show that for $\rho$'s which do not factor through $\Sp{2}(\ZZ)$ we have $\rmM^{(g)}(\sigma, \rho) = \{0\}$.  On the other hand, we have $\rmM^{(g)}(\sigma, \rho) = \{0\}$, if $\rho$ factors through $\Sp{g}(\ZZ)$ and $\sigma$ does not.

More concretely, this means that $\rmM^{(2)}_{k, l}(\rho) = \{ 0 \}$, if
\begin{enumerate}[(i)]
\item $k \in \ZZ$ and $\rho$ does not factor through $\Sp{2}(\ZZ)$, or
\item $k \in \frac{1}{2} \ZZ \setminus \ZZ$ and $\rho$ factors through $\Sp{2}(\ZZ)$.
\end{enumerate}
\end{remark}

\subsection{Structure theorems and dimension formulas }
\label{ssec:dimension-formulas}

Let $\rho$ be any (finite dimensional) representation of $\Sp{1}(\ZZ)$.  By results of Mason and Marks, $\bigoplus_{k \in \ZZ} \rmM^{(1)}_{k}(\rho)$ is a free module over $\bigoplus_{k \in \ZZ} \rmM^{(1)}_{k}$ of rank at most $\dim \rho$~\cite{MM10}.  Hence there is a polynomial $p_\rmg(\rho; t)$ with $p_\rmg(\rho; 1) \le \dim \rho$ such that
\begin{gather}
  \sum_{k \in \ZZ} \dim \rmM^{(1)}_{k}(\rho) t^k
=
  \frac{p_\rmg(\rho; t)}{(1 - t^4) (1 - t^6)}
\text{.}
\end{gather}

Fix a representation $\sigma$ of $\GL{2}(\CC)$ and a representation $\rho$ of $\Sp{2}(\ZZ)$ such that $\sigma(-I_2) \otimes \rho (-I_4)$ is trivial.  Then $\HS^{(2)} \times (V_\sigma \otimes V_\rho)$ descends to a vector bundle on $\Sp{2}(\ZZ) \backslash \HS^{(2)}$.  We can compute the asymptotic dimension of $\rmM^{(2)}_{\det^k \otimes\, \sigma}(\rho)$ using Hirzebruch-Riemann-Roch and the Lefschetz fixed point theorem.  We find that
\begin{gather}
\label{eq:asymptotic-dimensions-for-siegel-modular-forms}
  \rmM^{(2)}_{\det^k \otimes\, \sigma}(\rho)
\asymp
  \dim \sigma \otimes \rho \cdot \dim \rmM^{(2)}_k
\end{gather}
as $k \rightarrow \infty$, $k \in \ZZ$.  We give some details for the convenience of the reader.  Fix some normal congruence subgroup $\Gamma \subset \Sp{2}(\ZZ)$ such that there is a smooth compactification $\widetilde{X}_\Gamma$ of $X_\Gamma = \Gamma \backslash \HS^{(2)}$.  Write $X_{\Sp{2}(\ZZ)}$ for $\Sp{2}(\ZZ) \backslash \HS^{(2)}$, and $\widetilde{X}_{\Sp{2}(\ZZ)}$ for $(\Sp{2}(\ZZ) \slashdiv \Gamma) \backslash \widetilde{X}_\Gamma$.  Fix a line bundle~$E$ on $\widetilde{X}_{\Sp{2}(\ZZ)}$ of dimension~$d$.  The line bundle $\det$ is ample, so that $\rmH^i(\widetilde{X}_\Gamma, E \otimes \det^k) = 0$ for $i \ne 0$ and sufficiently large~$k$.  For such~$k$, Hirzebruch-Riemann-Roch gives
\begin{gather*}
  \dim \rmH^0(\widetilde{X}_\Gamma, E \otimes \det{}^k)
=
  \int_{\widetilde{X}_\Gamma} \mathop{\rm ch}(E) \mathop{\rm ch}(\det{}^k) \mathop{\rm td}(\widetilde{X}_\Gamma)
\text{,}
\end{gather*}
where ${\rm ch}(E)$ and ${\rm ch}(\det{}^k)$ are the Chern characters of $E$ and $\det{}^k$, respectively, and ${\rm td}(\widetilde{X}_\Gamma)$ is the Todd polynomial of $\widetilde{X}_\Gamma$.  The $i$th degree of ${\rm ch}(\det{}^k) = {\rm ch}(\det)^k$ has asymptotic $c_i k^i$ for some constant $c_i$.  Since $\det$ is positive, we have $c_i > 0$.  On the other hand the lowest degree of ${\rm ch}(E)$ equals $d$, and the lowest degree of ${\rm ch}(X_\Gamma)$ is $1$.  Hence we have
\begin{gather*}
  \dim \rmH^0(\widetilde{X}_\Gamma, E \otimes \det{}^k)
\asymp
  d \int_{X_\Gamma} c_3
\text{.}
\end{gather*}
In order to pass to $X_{\Sp{2}(\ZZ)}$, we use the Lefschetz fixed point theorem and the Koecher principle.  We have
\begin{align*}
&
  \dim \rmH^0 \big( X_{\Sp{2}(\ZZ)}, E \otimes \det{}^k \big)
=
  \dim \rmH^0 \big( \widetilde{X}_{\Sp{2}(\ZZ)}, E \otimes \det{}^k \big)
\\[4pt]
=&
  \frac{1}{\# (\Sp{2}(\ZZ) \slashdiv \Gamma)}
  \sum_{g \in \Sp{2}(\ZZ) \slashdiv \Gamma}
  \tr\big( g |_{ \rmH^0(\widetilde{X}_\Gamma, E \otimes \det{}^k) } \big)
\text{.}
\end{align*}
Since $\tr\big( g |_{ \rmH^0(\widetilde{X}_\Gamma, E \otimes \det{}^k)} \big)$ is bounded in~$k$ by by, roughly, the number of fixed points of~$g$, if $g \ne 1$, the stated asymptotic of dimensions follows when applying our considerations to $E = \HS^{(2)} \times (V_\rho \otimes V_{\sigma})$.

\section{Jacobi forms}
\label{sec:jacobi-forms}

Let
\begin{gather*}
  \HS^\rmJ
=
  \HS^{(1)} \times \CC
\end{gather*}
be the Jacobi upper half space.  The extended Jacobi group
\begin{gather*}
  \Gamma^\rmJ
=
  \SL{2}(\ZZ) \ltimes \ZZ^2 \mathop{\widetilde \times} \ZZ
\end{gather*}
has group law
\begin{gather*}
  \big(\gamma, (\lambda, \mu), \kappa \big)
  \cdot
  \big(\gamma', (\lambda', \mu'), \kappa' \big)
=
  \big( \gamma \gamma',\,
        (\lambda, \mu) \gamma' + (\lambda', \mu'),\,
        \kappa + \kappa'  + \lambda \mu' - \mu \lambda'\big)
\text{.}
\end{gather*}
It can be viewed as a subgroup of $\Sp{2}(\ZZ)$ via the embedding
\begin{gather*}
  \Big( \left(\begin{smallmatrix} a & b \\ c & d \end{smallmatrix}\right), (\lambda, \mu), \kappa \Big)
\mapsto
  \begin{pmatrix}
  a & 0 & b   & 0 \\
  0 & 1 & 0 & 0 \\
  c & 0 & d   & 0 \\
  0 & 0 & 0   & 1
  \end{pmatrix}
  \begin{pmatrix}
  1 & 0 & 0   & \mu \\
  \lambda & 1 & \mu & \kappa \\
  0 & 0 & 1   & -\lambda \\
  0 & 0 & 0   & 1
  \end{pmatrix}
\text{.}
\end{gather*}
We denote the double cover of $\Gamma^\rmJ$ by $\tGamma^\rmJ$.  Its elements are pairs $(\gamma^\rmJ, \omega)$ with $\gamma^\rmJ \in \Gamma^\rmJ$ and $\omega : \HS^{(1)} \rightarrow \tGL{1}(\CC)$ a holomorphic square root of $c \tau + d$.

The action of $\Gamma^\rmJ$ (and thus $\tGamma^\rmJ$) on $\HS^\rmJ$ is given by
\begin{gather*}
  \gamma^\rmJ (\tau, z)
=
  \big(\gamma, (\lambda, \mu), \kappa \big) (\tau, z)
=
  \big( \gamma \tau, \frac{z + \lambda \tau + \mu}{c \tau + d} \big)
\text{.}
\end{gather*}

For $m \in \QQ$ and $x \in \CC$, we set $e^m(x) = \exp(2 \pi i m\, x)$.  Given $k \in \frac{1}{2}\ZZ$, $m \in \QQ$ and a representation $\rho$ of $\tAff{1}(\CC) \times \tGamma^\rmJ$ with representation space~$V_\rho$, there is a slash action
\begin{multline*}
  \big( \phi \big|^\rmJ_{k, m, \rho} \, (\gamma^\rmJ, \omega) \big) (\tau, z)
\\[2pt]
=
  \omega(c \tau + d)^{-2k}\,
  e^m\big( \frac{-c (z + \lambda \tau + \mu)^2}{c \tau + d} + 2 \lambda z + \lambda^2 \tau + \kappa \big) \,
\\
  \rho\big( (\omega(\tau), c(z + \mu) - d \lambda), (\gamma^\rmJ, \omega) \big)^{-1}\,
  \phi( \gamma^\rmJ (\tau, z) )
\end{multline*}
on functions $\phi :\, \HS^\rmJ \rightarrow V_\rho$.  We will often use representations of $\tGamma^\rmJ$ and $\tAff{1}(\CC)$, which we then consider as representations of $\tAff{1}(\CC) \times \tGamma^\rmJ$ by tensoring with the trivial representation of $\tGamma^\rmJ$ and $\tAff{1}(\CC)$, respectively.  Note that $\big|_{k, m, \det^{k'} \otimes \rho} = \big|_{k + k', m, \rho}$ for any $k' \in \frac{1}{2} \ZZ$.


\begin{definition}
\label{def:jacobi-forms}
Let $\rho$ be a representation of $\tAff{1}(\CC) \times \tGamma^\rmJ$ with representation space~$V_\rho$.  A holomorphic function $\phi :\, \HS^\rmJ \rightarrow V_\rho$ is called a Jacobi form of weight~$k$, index~$m$, and type~$\rho$ if
\begin{enumerate}
\item \label{def:jacobi-forms:it:invariance}
  $\phi \big|^\rmJ_{k, m, \rho} \, (\gamma^\rmJ, \omega) = \phi$ for all $(\gamma^\rmJ, \omega) \in \Gamma^\rmJ$.
\item \label{def:jacobi-forms:it:growths}
  $\phi(\tau, \alpha \tau + \beta)$ is bounded (with respect to any norm on $V_\rho$) as $y \rightarrow \infty$ for all $\alpha, \beta \in \QQ$.
\end{enumerate}
\end{definition}
\noindent We denote the space of Jacobi forms of weight~$k$, index~$m$, and type~$\rho$ by $\rmJ_{k, m}(\rho)$.  If $\rho$ is trivial, we suppress it.  Note that $J_{k, m} = \{0\}$, if $m \in \QQ \setminus \ZZ$.  Except for Proposition~\ref{prop:fourier-jacobi-expansion-for-siegel-modular-forms}, in this section, we restrict our attention to $k \in \ZZ$.

The last component $\ZZ$ of $\Gamma^\rmJ$ is central.  Hence for irreducible representations $\rho$ it acts by scalars.  We have the following connection between $m$ and $\rho$.
\begin{proposition}
\label{prop:vanishing-of-jacobi-forms-by-rho}
Let $k \in \ZZ$, $m \in \QQ$.  If $\rho :\, \Aff{1}(\CC) \otimes \Gamma^\rmJ \rightarrow \GL{}(V_\rho)$ is irreducible and $\rho\big( I_2, (0,0), 1 \big) \ne \exp(2 \pi i\, m)$, then $\rmJ_{k, m}(\rho) = \{ 0 \}$.
\end{proposition}
\begin{proof}
This follows when inspecting the definition of $\big|^\rmJ_{k, m, \rho}$.
\end{proof}

Jacobi forms have a Fourier expansion
\begin{gather*}
  \phi(\tau, z)
=
  \sum_{\substack{ 0 \le n \in \QQ \\ r \in \QQ }}
  c(\phi; n, r)\, q^n \zeta^r
\text{.}
\end{gather*}
By Condition~\ref{def:jacobi-forms:it:invariance}, $c(\phi; n, r)$ only depends on $4 m n - r^2$ and $r \pmod{(2 m \ZZ)}$.  From Condition~\ref{def:jacobi-forms:it:growths}, one infers that $c(\phi; n, r) = 0$, if $4 m n - r^2 < 0$.  This is actually equivalent to Condition~\ref{def:jacobi-forms:it:growths}.  The vanishing order of $\phi \in \rmJ_{k, l}(\rho)$ is defined as
\begin{gather*}
  \ord\, \phi
=
  \min\{ n \in \QQ \,:\, c(\phi; n, r) \ne 0\}
\text{.}
\end{gather*}
The spaces of Jacobi forms with vanishing order at least $0 \le d \in \QQ$ play a central role in Section~\ref{sec:formal-fj-expansions}.  Define
\begin{gather}
\label{eq:def:jacobi-forms-with-vanishing-order}
  \rmJ_{k,m}(\rho)[d]
=
  \big\{ \phi \in \rmJ_{k,m}(\rho) \,:\, \ord\, \phi \ge d \big\}
\text{.}
\end{gather}

The definitions of maps $\cD_\nu$ that were given on page~29 of~\cite{EZ85} carry over to our setting, if $\rho$ is trivial on $\Aff{1}(\CC)$.  Given a Jacobi form $\phi$ of even weight~$k$, index~$m$, and type $\rho :\, \Gamma^\rmJ \rightarrow \GL{}(V_\rho)$, set
\begin{gather*}
  \cD_{2 \nu} (\phi) \text{,}
:=
  \sum_{0 \le \mu \le \nu}
  \frac{(-1)^\mu \, (2 \nu)! \, (k + 2 \nu - \mu - 2)!}
       {\mu! \, (2 \nu - 2 \mu)! \, (k + \nu - 2)!} \;
  \big( \frac{\partial_z}{2 \pi i} \big)^{2 \nu - 2 \mu}
  \big( \frac{\partial_\tau}{2 \pi i} \big)^\mu \phi
\text{,}
\end{gather*}
and
\begin{gather*}
  \cD_{2 \nu + 1} (\phi) \text{,}
:=
  \sum_{0 \le \mu \le \nu}
  \frac{(-1)^\mu \, (2 \nu + 1)! \, (k + 2 \nu - \mu - 1)!}
       {\mu! \, (2 \nu + 1 - 2 \mu)! \, (k + \nu - 1)!} \;
  \big( \frac{\partial_z}{2 \pi i} \big)^{2 \nu + 1 - 2 \mu}
  \big( \frac{\partial_\tau}{2 \pi i} \big)^\mu \phi
\text{.}
\end{gather*}
Note that either $\cD_{2 \nu}(\phi)$ or $\cD_{2 \nu + 1}(\phi)$ is zero depending on $k$ and $\rho$.  Combining the above maps, we define $\cD(k, \rho)$ for irreducible $\rho$.
\begin{gather*}
  \cD(k, \rho)
=
  \begin{cases}
    \cD_0 \oplus \cdots \oplus \cD_{2 \lfloor m \rfloor } 
  \text{,}
  &
    \text{if $(-1)^k \rho (-I_2)$ is trivial;}
  \\
    \cD_1 \oplus \cdots \oplus \cD_{2 \lfloor m \rfloor - 1} 
  \text{,}
  &
    \text{otherwise.}
  \end{cases}
\text{.}
\end{gather*}
If $\rho = \bigoplus_i \rho_i$ decomposes into irreducible components, then we set 
\begin{gather*}
  \cD(k, \rho)
=
  \bigoplus_i \cD(k, \rho_i)
\text{.}
\end{gather*}

For irreducible $\rho$, it is not hard to see that the image of $\cD(k, \rho)$ restricted to $\rmJ_{k, m}(\rho)[d]$ is contained in $\bigoplus_{\nu = 0}^{\lfloor m \rfloor} \rmM_{k + 2 \nu}(\rho)[d]$ or $\bigoplus_{\nu = 0}^{\lfloor m \rfloor - 1} \rmM_{k + 2 \nu + 1}(\rho)[d]$.

The next theorem is mostly due to Eichler and Zagier~\cite{EZ85}.  We simply adapt it to the case of vector valued Jacobi forms.
\begin{theorem}
\label{thm:jacobi-taylor-coefficients-injective}
Let $k \in \ZZ$, $m \in \QQ$, and suppose that $\rho :\, \Aff{1}(\CC) \otimes \Gamma^\rmJ \rightarrow \GL(V_\rho)$ is trivial on $\Aff{1}(\CC)$.  Then the map $\cD(k, \rho)$ is injective.
\end{theorem}
\begin{proof}
We may restrict to the case of irreducible $\rho$.  Further, we only treat the case of trivial $(-1)^k \rho(-I_2)$.  The proof in the other case works completely analogously.

Choose $0 < N \in \ZZ$ such that $\rho$ is trivial on $(N \ZZ)^2 \mathop{\widetilde \times} N \ZZ \subseteq \Gamma^\rmJ$.  In particular, we have $N m \in \ZZ$.  Given any Jacobi form $\phi$ of weight~$k$ and type $\rho :\, \Gamma^\rmJ \rightarrow V_\rho$, odd Taylor coefficients of $\phi$ vanish by assumption on $(-1)^k \rho(-I_2)$.  This follows directly from the Jacobi transformation law applied to $-I_2 \in \SL{2}(\ZZ)$.

Observe that $\cD(k, \rho)$ injectively maps the set of all $2 \nu$th Taylor coefficient ($0 \le \nu \le \lfloor m \rfloor$) to $\bigoplus_{\nu = 0}^{\lfloor m \rfloor} \rmM_{k + 2 \nu}(\rho)$.  If $\cD(k, \rho)( \phi ) = 0$, we find that the first $2 \lfloor m \rfloor + 1 \ge 2 m$ Taylor coefficients of $f \mathop{\circ} \phi$ vanish for any linear functional~$f$ on $V_\rho$.  On the other hand, $z \mapsto f \mathop{\circ} \phi(\tau, Nz)$ is a holomorphic elliptic function of index~$N^2 m \in \ZZ$.  By assumption it has at least $N^2 (2 \lfloor m \rfloor + 2)$ zeros in the range $z = \alpha + \beta \tau$, $\alpha, \beta \in [0, 1)$.  Consequently, $f \mathop{\circ} \phi = 0$, and since $f$ was any functional on $V_\rho$, we have $\phi = 0$.
\end{proof}

Consider the restriction of $\sigma_l$ to $\Aff{1}(\CC) \subset \GL{2}(\CC)$ (see Equation~(\ref{eq:inclusion-aff-gl})) and denote it by $\sigma_l$ as well.  For $m \in \ZZ$, it was shown in~\cite{IK11} that
\begin{gather*}
  \rmJ_{k, m}(\sigma_l)
\cong
  \bigoplus_{i = 0}^l \rmJ_{k + i, m}
\text{.}
\end{gather*}
This isomorphism is induced by differential operators with constant coefficients as stated on page~786 of~\cite{IK11}.  Hence it generalized to arbitrary $m \in \QQ$ and to any representation $\rho :\, \Gamma^\rmJ \rightarrow \GL{}(V_\rho)$.  We have
\begin{gather*}
  \rmJ_{k, m}(\sigma_l \otimes \rho)[d]
\cong
  \bigoplus_{i = 0}^l \rmJ_{k + i, m}(\rho)[d]
\text{.}
\end{gather*}
Combining this with the statement of~Theorem~\ref{thm:jacobi-taylor-coefficients-injective}, we obtain
\begin{corollary}
\label{cor:jacobi-with-representation-taylor-coefficients-injective}
Let $k \in \ZZ$, $m \in \QQ$, and $\rho :\, \Gamma^\rmJ \rightarrow V_\rho$ be an irreducible unitary representation.  Then there is an injective map
\begin{gather*}
  \rmJ_{k, m}(\sigma_l \otimes \rho)[d]
\hookrightarrow
  \bigoplus_{i = 0}^l
  \begin{cases}
    \bigoplus_{\nu = 0}^{\lfloor m \rfloor} M_{k + i + 2 \nu}(\rho)[d]
  \text{,}
  &
    \text{if $(-1)^{k + i} \rho(-I_2)$ is trivial;}
  \\
    \bigoplus_{\nu = 0}^{\lfloor m \rfloor - 1} M_{k + i + 2 \nu + 1}(\rho)[d]
  \text{,}
  &
    \text{otherwise.}
  \end{cases}
\end{gather*}
\end{corollary}

We finish this section with a statement on the connection of Siegel modular forms and Jacobi forms.
\begin{proposition}
\label{prop:fourier-jacobi-expansion-for-siegel-modular-forms}
Let $\Phi \in \rmM^{(2)}_{k, l}(\rho)$, $k \in \frac{1}{2} \ZZ$, $l \in \ZZ$ be a Siegel modular form, where $\rho$ is a unitary representation of $\Sp{2}(\ZZ)$ (whose kernel has finite index).  Then for all $m \in \QQ$ the function
\begin{gather*}
  \phi_m(\tau, z)
=
  \sum_{n, r \in \QQ} c(\Phi; n, r, m) \, q^n \zeta^r
\end{gather*}
is a Jacobi form of weight~$k$, index~$m$, and type~$\sigma_l \otimes \rho$.

In particular, we have a Fourier Jacobi expansion
\begin{gather*}
  \sum_{0 \le m \in \QQ}
  \phi_m(\tau, z) \, q^{\prime\, m}
\text{.}
\end{gather*}
\end{proposition}
\begin{proof}
This is a straight forward verification.
\end{proof}

\section{Formal Fourier Jacobi expansions}
\label{sec:formal-fj-expansions}

Given a Siegel modular forms $\Phi \in \rmM^{(2)}_{k, l}(\rho)$, there is a Fourier Jacobi expansion as discussed at the end of the previous section.  The notion of formal Fourier Jacobi expansions mimics this.
\begin{definition}
Fix $k \in \frac{1}{2}\ZZ$, $l \in \ZZ$, and a unitary representation $\rho$ of $\tSp{2}(\ZZ)$ (whose kernel has finite index).  Let
\begin{gather*}
  \Phi
=
  \sum_{0 \le m \in \QQ} \phi_m \, q^{\prime\, m}
\in
  \bigoplus_{m} \rmJ_{k,\, m}(\rho \otimes \sigma_l) \, q^{\prime\, m}
\text{.}
\end{gather*}
We call $\Phi$ a \emph{formal Fourier Jacobi expansion} of weight~$(k, l)$ and type~$\rho$ if
\begin{gather*}
  c(\phi_m; n, r)
=
  (\det{}^k \otimes \sigma_l)^{-1}(S)\, \rho^{-1}(\rot(S))\, c(\phi_n; m, r)
\end{gather*}
for all $n, r \in \QQ$, where $S = \left(\begin{smallmatrix} 0 & 1 \\ 1 & 0 \end{smallmatrix}\right)$.
\end{definition}
\noindent Denote the space of such expansions by $\rmF\rmM^{(2)}_{k,l}(\rho)$.

\begin{proposition}
\label{prop:formal-fourier-jacobi-determination-of-coefficients}
Suppose that $\Phi$ and $\Psi$ are formal Fourier Jacobi expansions of same weight and type.  If $\phi_m = \psi_m$ for $0 \le m < m'$ for some $0 < m' \in \QQ$, then $\phi_{m'} - \psi_{m'} \in \rmJ_{k, m'}(\rho \otimes l)[m']$.
\end{proposition}
\begin{proof}
For any $0 \le n < m'$ and any $r \in \ZZ$, we have
\begin{align*}
&\,
  c(\phi_{m'} - \psi_{m'}; n, r)
=
  c(\phi_{m'}; n, r) - c(\psi_{m'}; n, r)
\\
=&\,
  (\det{}^k \otimes \sigma_l)^{-1}(S)\, \rho^{-1}(\rot(S))
  \big( c(\phi_{n}; m', r) - c(\psi_{n}; m', r) \big)
=
  0
\text{.}
\end{align*}
\end{proof}

The idea of proof of the next theorem is very much inspired by~\cite{IPY12}, and was initially brought up by Akoi~\cite{Ao00}.
\begin{theorem}
\label{thm:generating-series-for-formal-fourier-jacobi}
We have
\begin{gather*}
  \dim \rmF\rmM^{(2)}_{k,l}(\rho)
\le
  \dim \rmM^{(2)}_{k,l}(\rho) + O(k^2)
\end{gather*}
as $k \rightarrow \infty$, $k \in 2 \ZZ$.
\end{theorem}

\begin{corollary}
\label{cor:dimension-estimates-for-formal-fourier-jacobi}
There is $0 < k_0 \in \ZZ$ such that
\begin{gather*}
  \dim \rmF\rmM_{k,l}(\rho) < 2 \dim \rmM_{k, l}(\rho)
\end{gather*}
for all $k > k_0$.
\end{corollary}

\begin{proof}[{Proof of Theorem~\ref{thm:generating-series-for-formal-fourier-jacobi}}]
By Proposition~\ref{prop:fourier-jacobi-expansion-for-siegel-modular-forms} and~\ref{prop:formal-fourier-jacobi-determination-of-coefficients}, we have
\begin{gather*}
  \rmF\rmM^{(2)}_{k,l}(\rho)
\hookrightarrow
  \bigoplus_{0 \le m \in \QQ}\rmJ_{k, m}(\sigma_l \otimes \rho)[m]
\text{.}
\end{gather*}
We are hence reduced to estimating the dimension of the right hand side.  Corollary~\ref{cor:jacobi-with-representation-taylor-coefficients-injective} says that there is a map
\begin{gather*}
  \rmJ_{k,m} (\sigma_l \otimes \rho)[m]
\hookrightarrow
  \bigoplus_{i = 0}^l \rmJ_{k + i, m}(\rho)[m]
\text{.}
\end{gather*}
Plug this into the previous equation.  Then Equation~(\ref{eq:asymptotic-dimensions-for-siegel-modular-forms}) shows that it suffices to treat the case~$l = 0$.

By Proposition~\ref{prop:vanishing-of-jacobi-forms-by-rho}, we need to consider
\begin{gather*}
  \bigoplus_{0 \le m_0 < 1} \;
  \bigoplus_{0 \le m \in m_0 + \ZZ} \rmJ_{k, m}(\sigma_l \otimes \rho(m_0))[m]
\end{gather*}
where $\rho = \bigoplus_{m_0} \rho(m_0)$ and $\rho(m_0)\big( I_2, (0,0), 1 \big) = \exp(2 \pi i\, m_0)$.

We consider the generating function of this series, and adapt the calculations presented in Section~4 of~\cite{IPY12}.  The change of order of summation is justified because all coefficients are positive.
\begin{align*}
&
  \sum_{2 \isdiv k \, \in \ZZ} \; \sum_{0 \le m_0 < 1} \;
  \sum_{m = 0}^\infty \sum_{\nu = 0}^{m} \dim \rmM_{k + 2 \nu}(\rho(m_0))[m_0 + m]\, t^k
\\[2pt]
=
&
  \sum_{0 \le m_0 < 1} \; \sum_{m = 0}^\infty \sum_{\nu = 0}^m
  \Big( \sum_{2 \isdiv k \,\in \ZZ} \dim \rmM_{k + 2 \nu - 12 m}(\rho(m_0))[m_0] \, t^{k + 2 \nu - 12 m} \Big)
  t^{12 m - 2 \nu}
\\[2pt]
\le
&
  \frac{\sum_{0 \le m_0 < 1} p_\rmg(\rho(m_0); t)}{(1 - t^4)(1 - t^6)} \;
  \sum_{m = 0}^\infty \sum_{\nu = 0}^{m} t^{12 m - 2 \nu}
=
  \frac{p_\rmg(\rho; t)}{(1 - t^4)(1 - t^6)} \;
  \sum_{m = 0}^\infty \sum_{\nu = 0}^{m} t^{12 m - 2 \nu}
\text{.}
\end{align*}
Here, $\le$  means that every coefficient with respect to $t$ on the left hand side is less than or equal to the corresponding coefficient on the right hand side.  The double sum equals
\begin{gather*}
  \sum_{\nu = 0}^\infty \sum_{m = \nu}^\infty t^{12 m - 2 \nu}
=
  \frac{1}{1 - t^{12}} \sum_{\nu = 0}^\infty t^{12 \nu - 2 \nu}
=
  \frac{1}{(1 - t^{10})(1 - t^{12})}
\text{.}
\end{gather*}
This proves the statement, since the generators of $\bigoplus_{2 \isdiv k} \rmM^{(2)}_k$ have weight $4, 6, 10$, and~$12$.
\end{proof}

\begin{lemma}
\label{la:non-vanishing-forms-of-half-integral-weight}
There is a vector valued Siegel modular form of
\begin{enumerate}[(i)]
\item odd weight
\item half integral weight $k \not \in \ZZ$
\end{enumerate}
 whose components have no common zero.
\end{lemma}
\begin{proof}
We prove the second statement only.  The first follows along the same lines.

Fix some Siegel modular form~$\Phi$ of half integral weight~$k$ for some finite index subgroup of $\tSp{2}(\ZZ)$.  For example, any of the theta constants works~\cite{Fr91}.  There is $0 < j \in \ZZ$, and a set $\{ \alpha_i \}_{1 \le i \le j} \subset \QQ$ such that the shifts $\Phi(Z + \alpha_i)$, $1 \le i \le j$ of $\Phi$ have no zero in common.  Each such shift is a Siegel modular form for some subgroup $\Gamma_i \subseteq \tSp{2}(\ZZ)$.  We obtain vector valued Siegel modular forms
\begin{gather*}
  \Phi_i(Z)
=
  \sum_{\gamma \in \Gamma_i \backslash \tSp{2}(\ZZ)}
    \Phi(Z + \alpha_i) \big|^{(2)}_k\, \gamma\;
    \frake_\gamma
\end{gather*}
for the induced representations $\rho_i = {\rm Ind}_{\Gamma_i}^{\Sp{2}(\ZZ)}(\mathbbm{1})$ with representation space $V_{\rho_i} = \lspan\big( \frake_\gamma \,:\, \gamma \in \Gamma_i \backslash \Sp{2}(\ZZ) \big)$.  Then $\big( \Phi_i \big)_{1 \le i \le j}$ is a vector valued Siegel modular form of type $\bigoplus_{1 \le i \le j} \rho_i$ with the desired property.
\end{proof}

\begin{theorem}
\label{thm:formal-fourier-jacobi-series-equal-siegel-modular-forms}
For every $k \in \frac{1}{2} \ZZ$, $0 \le l \in \ZZ$, and every unitary representation~$\rho$ of $\tSp{2}(\ZZ)$ (whose kernel has finite index), we have $\rmF\rmM^{(2)}_{k, l}(\rho) = \rmM^{(2)}_{k, l}(\rho)$.
\end{theorem}
\begin{proof}
We first prove the case of $k \in 2 \ZZ$.  By the remark after Definition~\ref{def:siegel-modular-forms}, $\rho$ factors through $\Sp{2}(\ZZ)$.  Suppose that
\begin{gather*}
  \Phi
=
  \sum_{0 \le m \in \ZZ} \phi_m \, q^{\prime\, m}
\in
  \rmF\rmM^{(2)}_{k, l}(\rho)
\text{.}
\end{gather*}
Choose $k_0$ as in Corollary~\ref{cor:dimension-estimates-for-formal-fourier-jacobi}, and consider the space
\begin{gather*}
  \rmM^{(2)}_{k_0} \cdot \Phi
\subseteq
  \rmF\rmM^{(2)}_{k + k_0, l}(\rho)
\text{.}
\end{gather*}
By choice of $k_0$, the intersection of $\rmM^{(2)}_{k_0} \cdot \Phi$ and $\rmM^{(2)}_{k + k_0, l}(\rho)$ is not empty.  Pick some $\Psi \in \rmM^{(2)}_{k_0}$ such that $\Psi \Phi \in \rmM^{(2)}_{k + k_0, l}(\rho)$.  Next, apply Theorem~\ref{thm:basic_regularity} to find that $\Phi \in \rmM_{k, l}(\rho)$.

Now, suppose that $k \in \frac{1}{2}\ZZ$.  Choose a Siegel modular form $\Psi$ as in Lemma~\ref{la:non-vanishing-forms-of-half-integral-weight} such that $\Phi \otimes \Psi$ has even weight.  By the above, it is a Siegel modular form.  Write $\Psi_i$ for the components of $\Psi$.  The meromorphic Siegel modular form $(\Phi \Psi_i) \slashdiv \Psi_i$ is independent of~$i$, since $(\Phi \Psi_i) \Psi_{i'} = (\Phi \Psi_{i'}) \Psi_{i}$ for all~$i, i'$.  By choice, the components of $\Psi$ have no common zero.  Therefore, $(\Phi \Psi_i) \slashdiv \Psi_i$ is holomorphic and posses a Fourier Jacobi expansion.  This expansions equals $\Phi$, because multiplication by $\Psi_i$ is injective on formal Fourier Jacobi expansions.
\end{proof}

\section{Regularity for meromorphic Siegel modular forms}
\label{sec:regularity}

In this section we prove regularity of meromorphic Siegel modular forms which, in a sense made precise below, have a holomorphic Fourier-Jacobi expansion.  Note that all considerations are independent of Section~\ref{sec:formal-fj-expansions}, where we use Theorem~\ref{thm:siegel-regularity} to establish Theorem~\ref{thm:formal-fourier-jacobi-series-equal-siegel-modular-forms}.

\begin{theorem}
\label{thm:siegel-regularity}
Let $\Phi = \sum_{0 \le m \in \QQ} \phi\, q^{\prime\, m} \in \rmF\rmM^{(2)}_{k, l}(\rho)$ and $\Psi = \sum_{0 \le m \in \QQ} \psi\, q^{\prime\, m} \in \rmM^{(2)}_{k'}$.  Suppose that $\Phi \Psi \in \rmM^{(2)}_{k + k', l}(\rho)$.  Then $\sum_{0 \le m \in \QQ} \phi\, q^{\prime\, m}$ is convergent and $\Phi \in \rmM^{(2)}_k$.
\end{theorem}
\begin{proof}
We show that $\Phi$ has no singularity on the boundary of the Satake compactification.  By the Koecher principle, this implies the statement.

Write $D = \{ x \in \CC :\, |x| < 1 \}$ and $\dot{D} = D \setminus \{0\}$.  We start be reducing our considerations to a map $\Phi^* :\, D^3 \rightarrow \CC\bbP^1$.  To achieve this, we use a special case of the approach taken in~\cite{Br04}.  Let $\Gamma \subset \Sp{2}(\ZZ)$ be a torsion free subgroup such that the restriction of $\rho$ to~$\Gamma$ is trivial.  Set $X_\Gamma = \Gamma \backslash \HS^{(2)}$, and define $Y_\Gamma$ as the Satake compaktification of~$X_\Gamma$.  Employing Hironaka's theory to the boundary $\partial X_\Gamma \subset Y_\Gamma$ and the divisor of~$\Psi$, we find a resolution of singularities ${\widetilde Y}_\Gamma \rightarrow Y_\Gamma$.  For all $(\tau_0, z_0) \in \HS^\rmJ$, this yields a holomorphic map $D^2 \times \dot{D} \rightarrow \HS^{(2)}$, whose image contains all $Z = \left(\begin{smallmatrix}\tau_0 & z_0 \\ z_0 & \tau'\end{smallmatrix}\right)$ with $\Im \tau' > y'_0$ for some $y'_0 \in \RR$.  The components correspond to $(\tau, z)$ and $q'$.  After possibly shrinking the image and preimage, and after composing with a suitable biholomorphic map on $D^2 \times \dot{D}$, we can assume that the pushforward of the derivative with respect to the canonical coordinate on~$\dot{D} \subset D^2 \times \dot{D}$ is mapped to $\partial_{q'}$ on $\HS^{(2)}$.  By abuse of notation we write $q'$ for the corresponding variable.  Denote by~$E$ the divisor $\{q' = 0\} \subset D^2 \times D$, which is the preimage of $\partial X_\Gamma \subset Y_\Gamma$.

Since $\Phi \Psi$ is a Siegel modular form, it extends from $X_\Gamma$ to $Y_\Gamma$, and so does~$\Psi$.  Therefore the pullback $\Psi^* : D^2 \times \dot{D} \rightarrow \CC\bbP^1$ extends to $D^2 \times D$.  Its divisor~, $\div(\Psi^*)$, is normal crossing.  It decomposes as $\div_1(\Psi) + \div_2(\Psi)$, where $\div_1(\Psi)$ is supported on $E$, and $\div_2(\Psi) + E$ is normal crossing.  Write $E^\rmJ$ for the intersection $|\div_2(\Psi)| \cap E$, where we mean by $|\div_2(\Psi)|$ the support of~$\div_2(\Psi)$.  By the above, $E^\rmJ$ is smooth and has codimension~$1$, if it is not empty.

Fix $(\tau_0, z_0) \in \HS^\rmJ \setminus E^\rmJ$.  There is a neighborhood $U^\rmJ \subseteq \HS^\rmJ$ of $(\tau_0, z_0)$ and $y'_0 \in \RR$ such that $\Psi(Z) \ne 0$ for all $(\tau, z) \in U^\rmJ$ and $\Im(\tau') > \Im(\tau'_0)$.  Therefore we have
\begin{gather*}
  \partial_{q'}^m\big|_{q' = 0} \big( (\Phi \Psi)^* \slashdiv \Psi^* \big)
=
  (\phi_{m})^*
\end{gather*}
as a convergent series.  By the assumptions, $(\phi_m)^*$ extends to $\HS^\rmJ$.  By applying to the above equation linear functionals on the pullback of $V_l \otimes V_\rho$, by then restricting to suitable smooth subvarieties of $\HS^\rmJ$, and by finally using Theorem~\ref{thm:basic_regularity}, we establish the statement.
\end{proof}

The remainder of this section is of purely function theoretic nature.  The proof of Theorem~\ref{thm:basic_regularity} builds up on Lemma~\ref{la:basic_regularity:general_case} and~\ref{la:basis_regularity:vanishing}, which summarize slightly involved computations with polynomials whose coefficients are meromorphic functions.

In the proofs of the next theorem and lemmas, we will repeatedly use the Pochhammer symbol
\begin{gather}
\label{eq:def:pochhammer-symbol}
  (a)_n
=
  \prod_{i = 0}^{n - 1} (a - i)
\text{,}
\end{gather}
which is defined for $a \in \CC$ and $0 \le n \in \ZZ$.

\begin{theorem}
\label{thm:basic_regularity}
Let $f(x, y)$ be a function that is meromorphic in a neighborhood~$U$ of $x = y = 0$ in $\CC^2$.  Assume the following:
\begin{enumerate}[(i)]
\item $\{ (x, y) \in U \,:\, y = 0 \}$ is not a polar divisor of~$f$.
\item No irreducible polar divisor of $f$ has a branching point at $x = y = 0$.
\item For all $0 \le n \in \ZZ$, the function $\partial_y^n\big|_{y = 0}\, f(x, y)$ admits a holomorphic continuation to a neighborhood of~$x = 0$.
\end{enumerate}
Then $f$ is holomorphic in $x = y = 0$.
\end{theorem}
\begin{proof}
By the assumptions on~$f$ and the theory of meromorphic functions presented, for example, in~\cite{Ni01} (following work of Oka), we express~$f$ in a sufficiently small neighborhood of $x = y = 0$ as follows:
\begin{gather}
\label{eq:thm:basic_regularity}
  f(x, y)
=
  {\td f}(x, y)\,
  \frac{\sum_{i = 0}^{d_P} P_{i}(x) y^i}
       {\prod_{j = 1}^{n_Q} \big(Q_{j, 1}(x) y + Q_{j, 0}(x)\big)^{1 + m_j}}
\text{.}
\end{gather}
We have $0 \le n_Q, d_P, m_j \in \ZZ$ ($1 \le j \le n_Q$).  Functions that appear on the right hand side are holomorphic in some neighborhood of $x = y = 0$ or $x =0$, respectively.  They satisfy ${\td f}(0, 0) \ne 0$, $P_{d_P}(0) \ne 0$, $P_{i}(0) = 0$ ($0 \le i < d_P$), $Q_{j, 1}(0) \ne 0$, $Q_{j, 0}(0) = 0$.  We can assume that for $j \ne j'$, $-Q_{j, 0} \slashdiv Q_{j, 1} \ne -Q_{j', 0} \slashdiv Q_{j', 1}$ as functions in a neighborhood of~$x = 0$.  By dividing both sides by~${\td f}$, we can and will assume without loss of generality that ${\td f}(x, y) = 1$.  Write $P(x, y) = \sum_i P_i(x) y^i$ for the numerator of~\eqref{eq:thm:basic_regularity}.  We may assume that the numerator and denominator in~\eqref{eq:thm:basic_regularity} are coprime, which amounts to the assertion that for all~$1 \le j \le n_Q$, we have $P\big(x,\, -Q_{j,0}(x) \slashdiv Q_{j, 1}(x)\big) \ne 0$ as functions in a neighborhood of $x = 0$.

We will show by contradiction that $n_Q = 0$, that is, $f(x, y) = P(x, y)$ is holomorphic in $x = y = 0$.

For convenience, we suppress the dependence on~$x$.  Furthermore, we use exponents $0 \le e_j \in \ZZ$.  Away from $x = 0$, we have
\begin{align*}
&
  \partial_y^n \big|_{y = 0} f(\,\cdot\,, y)
\\[0.2em]
={}&
  \sum_{i = 0}^{d_P} \binom{n}{i} i! P_i\;
  \partial_y^{n - i}\big|_{y = 0}
  \Big( \prod_{j = 1}^{n_Q} \big(Q_{j, 1} y + Q_{j, 0}\big)^{-1 - m_j} \Big)
\\
={}&
  \sum_{i = 0}^{d_P} \binom{n}{i}i! P_i
   \sum_{\sum_j e_j = n - i} \binom{n - i}{e_1, e_2, \ldots} \prod_{j = 1}^{n_Q} (-1)^{e_j} (m_j + e_j)_{e_j}\, Q_{j, 1}^{e_j} Q_{j,0}^{-1 - m_j - e_j}
\\
={}&
  \Big( \prod_{j = 1}^{n_Q} Q_{j, 0}^{-1 - m_j} \Big)
  \sum_{i = 0}^{d_P} P_i
  \sum_{\sum_j e_j = n - i} \binom{m_j + e_j}{m_j}
  \prod_{j = 1}^{n_Q} (- Q_{j,1} \slashdiv Q_{j, 0})^{e_j}
\text{.}
\end{align*}
Assuming $n_Q \ne 0$, apply Lemma~\ref{la:basic_regularity:general_case} and~\ref{la:basis_regularity:vanishing} to obtain a contradiction.
\end{proof}

\begin{lemma}
\label{la:basic_regularity:base_case}
Let $H_j(x)$ ($1 \le j \le n_H$) be pairwise distinct meromorphic functions in a neighborhood of~$x = 0$.  Then for all $0 < n \in \ZZ$, we have
\begin{gather*}
  \sum_{\sum_{j = 1}^{n_H} e_j = n} \prod_{j = 1}^{n_H} H_j^{e_j}
=
  \sum_{j = 1}^{n_H} H_j^{n + n_H} \prod_{j' \ne j} (H_j - H_{j'})^{-1}
\text{,}
\end{gather*}
where the exponents~$e_j \in \ZZ$ are nonnegative, for all $j$.
\end{lemma}
\begin{proof}
Compute the product 
\begin{gather*}
  \prod_{j' \ne j} (H_j - H_{j'})
  \sum_{\sum_{j = 1}^{n_H} e_j = n} \prod_{j = 1}^{n_H} H_j^{e_j}
\end{gather*}
by repeatedly applying the formula
\begin{gather*}
  (H_{j'} - H_{j''})
  \sum_{\sum_{j = 1}^{n_H} e_j = n}\hspace{-.7em} H_j^{e_j}
=
  H_{j'}\!\! \sum_{\substack{\sum_j e_j = n \\ e_{j''} = 0}} \prod_{j = 1}^{n_H} H_j^{e_j}
  - H_{j''}\!\! \sum_{\substack{\sum_j e_j = n \\ e_{j'} = 0}} \prod_{j = 1}^{n_H} H_j^{e_j}
\text{.}
\qedhere
\end{gather*}
\end{proof}

\begin{lemma}
\label{la:basic_regularity:general_case}
Let $H_j(x)$ ($1 \le j \le n_H$) be pairwise distinct, nonzero meromorphic functions in a neighborhood of~$0$.  Fix a holomorphic function $P(x, y) = \sum_{i = 0}^{d_P} P_i(x) y^i$, and suppose that $P(x, H_j(x)^{-1}) \ne 0$ as functions of $x$, for all~$1 \le j \le n_H$.  Given nonnegative integers $m_j$, we have
\begin{gather*}
  \sum_{i = 0}^{d_P} P_i \sum_{\sum_{j = 1}^{n_H} e_j = n} \binom{e_j + m_j}{m_j} H_j^{e_j}
=
  \sum_{j = 1}^{n_H} H_j^n R_{j}(H_1, \ldots, H_{n_H}; n)
\end{gather*}
with rational functions $R_{j}$.  Moreover, the $R_j(H_1(x), \ldots, H_{n_H}(x); n)$ ($1 \le j \le n_H$) are nonzero polynomials in~$n$ with coefficients that are meromorphic functions in~$x$.
\end{lemma}
\begin{proof}
Write $\partial_j$ for the formal derivative with respect to~$H_j$.  Using Lemma~\ref{la:basic_regularity:base_case}, we find
\begin{align*}
  \sum_{\sum_{j = 1}^{n_H} e_j = n} \binom{e_j + m_j}{m_j} H_j^{e_j}
&=
  \Big( \prod_{j = 1}^{n_H} \frac{\partial_{H_j}^{m_j}}{m_j!} \Big)
  \Big(
  \sum_{\sum_{j = 1}^{n_H} e_j = n} \prod_{j = 1}^{n_H} H_j^{e_j + m_j}
  \Big)
\\
&=
  \Big( \prod_{j = 1}^{n_H} \frac{\partial_{H_j}^{m_j}}{m_j!} \Big)
  \Big(
  \big( \prod_{j = 1}^{n_H} H_j^{m_j} \big)
  \sum_{j = 1}^{n_H} H_j^{n + n_H} \prod_{j' \ne j} (H_j - H_{j'})^{-1}
  \Big)
\text{.}
\end{align*}
We plug this into the left hand side of the formula in the statement.  Note that the sum over~$i$ and the formal derivatives commute, since $P_i$, as a variable, is independent of $H_j$.
\begin{multline*}
  \Big( \prod_{j = 1}^{n_H} \frac{\partial_{H_j}^{m_j}}{m_j!} \Big)
  \Big(
  \sum_{j = 1}^{n_H}
  H_j^n\;
  \frac{P(H_j^{-1}) H_j^{n_H} \prod_{j'} H_{j'}^{m_{j'}}}
       {\prod_{j' \ne j} (H_j - H_{j'})}
  \Big)
\\
=
  \sum_{j = 1}^{n_H}
  H_j^n
  \sum_{l = 0}^{m_j}
  (n)_{m_j - l}
  H_j^{l - m_j}\;
  \frac{\partial_j^l}{m_j!} \Big(
  P(H_j^{-1}) H_j^{n_H + m_j}
  \Big( \prod_{j' \ne j} \frac{\partial_{j'}^{m_{j'}}}{m_{j'}!} 
        \frac{H_{j'}^{m_{j'}}}{(H_j - H_{j'})} \Big) \Big)
\text{.}
\end{multline*}
This proves the first assertion of the lemma.

To prove that the~$R_j$ do not vanish as polynomials in~$n$, we start by computing
\begin{align*}
&
  \frac{\partial_{j'}^{m_{j'}}}{m_{j'}!}
  \frac{H_{j'}^{m_{j'}}}{(H_j - H_{j'})}
=
  \frac{1}{m_{j'}!}
  \sum_{l = 0}^{m_{j'}} \binom{m_{j'}}{l}
  \frac{l! (m_{j'})_{m_{j'} - l} H_{j'}^l }{(H_j - H_{j'})^{1 + l}}
\\
={}&
  \frac{1}{(H_j - H_{j'})^{1 + m_{j'}}}
  \sum_{l = 0}^{m_{j'}} \binom{m_{j'}}{l}
  H_{j'}^l (H_j - H_{j'})^{m_{j'} - l}
=
  \frac{H_j^{m_{j'}}}{(H_j - H_{j'})^{1 + m_{j'}}}
\text{.}
\end{align*}
Therefore, we have
\begin{gather*}
  R_j(H_1, \ldots, H_{n_H}; n)
=
  n^{m_j} P(H_j^{-1}) H_j^{n_H} \prod_{j' \ne j} \frac{H_j^{m_{j'}}}{(H_j - H_{j'})^{1 + m_{j'}}}
  + O(n^{m_j - 1})
\text{,}
\end{gather*}
where $R_j$ is a polynomial in~$n$.  It suffices to notice that the leading coefficient, the coefficient of $n^{m_j}$, is nonzero as a function of $x$.  This proves the second part.
\end{proof}

\begin{lemma}
\label{la:basis_regularity:vanishing}
Let $H_j(x)$ ($1 \le j \le n_H$) be pairwise distinct, nonzero meromorphic functions in a neighborhood of $x = 0$ whose pole order at $x = 0$ is at least~$1$.  Further, let $R_j(x, n)$ ($1 \le j \le n_H$) be polynomials in $n$ whose coefficients are meromorphic in~$x$.

If for all $0 < n \in \ZZ$,
\begin{gather*}
  \sum_{j = 1}^{n_H} H_j^n R_j(x, n)
\end{gather*}
is regular at $x = 0$, then $R_j = 0$ for all $1 \le j \le n_H$.
\end{lemma}
\begin{proof}
Set $J = \{ 1 \le j \le n_J \,:\, R_j \ne 0\}$.  The statement is equivalent to $J = \emptyset$.  We assume that $J$ is not empty.

Let $o_H = -\min_{j \in J} (\ord_{x = 0} H_j)$ be the maximal pole order of~$H_j$ at~$x = 0$, and $J_H = \{ j \in J \,:\, -\ord_{x = 0} H_j = o_H \}$ be the set of indices for which this maximum is attained.  Further, let $d_R = \max_{j \in J_H} (\deg_n R_j)$ be the maximal degree of $R_j$ as a polynomial in $n$, and let $J_R = \{ j \in J_H \,:\, \deg_n R_j = d_R \}$.  By construction, $J \ne \emptyset$ implies $J_H \ne \emptyset$, which implies $J_R \ne \emptyset$.  Write $R_{j, d_R}$ for the $d_R$th coefficient of $R_j$.  We will establish the lemma by finding a contradiction.

Choose some $0 \le o_R \in \ZZ$ such that the coefficients of $x^{o_R} R_j(x, n)$ are holomorphic in~$x = 0$ for all $1 \le j \le n$.  We rewrite the regularity assumption:
\begin{gather*}
  \sum_{j = 1}^{n_H} H_j^n R_j(x, n)
=
  n^{d_R} x^{-n o_H - o_R}
  \sum_{j = 1}^{n_H} (x^{o_H} H_j)^n (n^{-d_R} x^{o_R} R_j(x, n))
\text{,}
\end{gather*}
must be regular at~$x = 0$.  In other words, for all $0 \le l < n_H$, we have
\begin{gather*}
  \sum_{j = 1}^{n_H} (x^{o_H} H_j)^l\,
  \big( (x^{o_H} H_j)^n ( x^{o_R} R_{j, d_R} (x) + O((n + l)^{-1}) ) \big)
\equiv
  0
  \pmod{x^{n o_H + o_R}}
\text{.}
\end{gather*}
The Vandermon matrix $((x^{o_H} H_j)^l)_{j, l}$ is invertible as a matrix of meromorphic functions.  The $0$th coefficient of $(x^{o_H} H_j)^n$ is nonzero for $j \in J_H \supseteq J_R$.  By choosing $n$ large enough, we therefore see that
\begin{gather*}
   x^{o_R} R_{j, d_R} (x)
\equiv
  O(n^{-1})
  \pmod{x^{{\td n}(n)}}
\end{gather*}
with ${\td n}(n) \rightarrow \infty$ as $n \rightarrow \infty$.  This shows $R_{j, d_R} = 0$ for all $j \in J_H$.  This contradicts the choice of $d_R$, and hence proves the lemma.
\end{proof}

\section{Generating functions for special cycles on Shimura varieties}
\label{sec:special-cycles}

Let $\cL$ be a lattice of signature $(n, 2)$ with attached quadratic form $q_\cL$.  We adopt the notation used in the introduction.  That is, we denote the dual of $\cL$ by $\cL^\#$, and write $\disc^r\, \cL = (\cL^\# \slashdiv \cL)^r$ for powers of the attached discriminant form.  Fix a subgroup $\Gamma$ of the orthogonal group $\Orth{}(\cL)$ that fixes the discriminant form $\disc\, \cL$ pointwise.  Write ${\rm Gr}^-(\cL \otimes \CC)$ for the Grassmannian of $2$-dimensional negative subspaces of $\cL \otimes \CC$.  There are rational quadratic cycles on the Shimura variety $X_\Gamma = \Gamma \backslash {\rm Gr}^-(\cL \otimes \RR)$.  Given a tuple of vectors $v = (v_1, \ldots, v_r)$ in $\cL \otimes \QQ$, let
\begin{gather*}
  Z(v)
=
  \{ W \in {\rm Gr}^-(\cL \otimes \RR) \,:\, \lspan(v) \perp W \}
\text{.}
\end{gather*}
This cycle is non-trivial, if $\lspan(v) \subseteq \cL \otimes \QQ$ is positive.

Kudla defined so-called special cycles on $X_\Gamma$~\cite{Ku97}.  Associate to each $v$ a moment matrix $q_\cL(v) = \frac{1}{2} ( \langle v_i, v_j \rangle_\cL )_{1 \le i,j \le r}$, where $\langle v, w \rangle_\cL = q_\cL(v + w) - q_\cL(v) - q_\cL(w)$.  Given a semi-positive definite matrix $0 \le T \in \MatT{r}(\QQ)$ and $\lambda \in \disc^r\, \cL$, set
\begin{gather*}
  \Omega(T, \mu)
=
  \big\{ v \in \mu + \cL^r \,:\, T = q_\cL(v) \big\}
\text{.}
\end{gather*}
The symmetries $\Gamma$ act on $\Omega(T, \mu)$, and there are finitely many orbits.  Define
\begin{gather*}
  Z(T, \mu)
=
  \sum_{v \in \Gamma \backslash \Omega(T, \mu)}
  Z(v)
\text{.}
\end{gather*}
All cycles of this form descend to cycles on $X_\Gamma$, which we also denote by $Z(T, \mu)$.  Writing $\rk(T)$ for the rank of~$T$, we find that $Z(T, \mu)$ is a $\rk(T)$-cycle, whose class in $\CH^{\rk(T)}(X_\Gamma)_\CC$ is denoted by $\{Z(T, \mu)\}$.

Let $\omega^\vee$ be the anti-canonical bundle on $X_\Gamma$, and write $\{ \omega^\vee \}$ for its class in $\CH^1(X_\Gamma)_\CC$.  Fix a linear functional on $\CH^r(X_\Gamma)$.  Define the formal Fourier expansion
\begin{gather}
  \Theta_{\Gamma, f}(Z)
=
  \sum_{\substack{ \mu \in \disc^r\, \cL \\ 0 \le T \in \MatT{r}(\QQ) }}
    f\big( \{Z(T, \mu)\} \cdot \{\omega^\vee \}^{r - \rk(T)} \big) \,
    \exp\big(2 \pi i\, \tr(T Z) \big) \; \frake_\mu
\text{.}
\end{gather}
Here, $\frake_\mu$ is a canonical basis vector of the representation space $\CC\big[ \disc^r\, \cL \big]$ of the Weil representation of the double cover $\tSp{r}(\ZZ)$ of $\Sp{r}(\ZZ)$.

We now restrict to the case $r = 2$.  In his thesis, Zhang~\cite{Zh09} proved that $\Theta_{\Gamma, f}(Z)$ is a formal Fourier Jacobi expansions of weight $1 + \frac{n}{2}$.  This is essentially Proposition~2.6 on page~22.  Note that he does not make use of Condition~1 of his Theorem~2.5 while proving this proposition.
\begin{remark}
The type of $\Theta_{\Gamma, f}$ stated in~\cite{Zh09} is slightly incorrect.  On page~20, Zhang affirms that, in his notation,
\begin{gather*}
  F(T, \ud{\lambda})
=
  \sqrt{\det(A)}^{n' + n} \sqrt{\det(A)}^{n' - n}\, F(A^\tr T A, \ud{\lambda} A)
\text{.}
\end{gather*}
He argues that this holds if $n'$ is even.  But actually, $F(T, \ud{\lambda})$ is basis independent.  Consequently, the required invariance does not hold if $4 \nisdiv n'$, while $n' = 2$ in our application.  This can be fixed easily, by considering the subrepresentation of the Weil representation~$\rho_{L, r}$ (in Zhang's notation), which comes with the matching sign.
\end{remark}

\begin{corollary}
For any subgroup $\Gamma \subseteq O(L)$ as above and every linear functional $f$ on the Chow group $\CH^2(X_\Gamma)$ the function $\Theta_{\Gamma, f}$ is a vector valued Siegel modular form of weight $1 + \frac{n}{2}$.
\end{corollary}
\begin{proof}
This follows when combining Zhang's results and Theorem~\ref{thm:formal-fourier-jacobi-series-equal-siegel-modular-forms}.
\end{proof}

Using this statement, we can reprove Theorem~3.1 of~\cite{Zh09}.
\begin{corollary}
The space
\begin{gather*}
  \lspan\big( \{Z(T, \mu)\} \,:\, 0 < T \in \MatT{2}(\QQ),\, \mu \in \disc^r\, \cL \big)
\subseteq 
  \CH^2(X_\Gamma)_\CC
\end{gather*}
is finite dimensional.
\end{corollary}
\begin{proof}
This follows from $\dim \rmM^{(2)}_{k}(\rho) < \infty$, which holds for all $k \in \frac{1}{2} \ZZ$ and any unitary representation $\rho$ of $\tSp{2}(\ZZ)$.
\end{proof}

\section{Computing Siegel modular forms via Fourier Jacobi expansions}
\label{sec:computing-via-fj-expansions}

In this section, we assume that the base field for all modular forms is~$\QQ$.  All statements concerning dimensions and Fourier expansions that we will make use of hold for any sufficiently large field $\QQ \subseteq K \subseteq \CC$.
If the reader wishes, he can adjust Algorithm~\ref{alg:compute-via-fj-expansion:compute-jacobi} accordingly.

We deal with truncated Fourier expansions of Jacobi forms, Siegel modular forms and formal Fourier Jacobi expansions.  We start be defining appropriate index sets.  Given $0 < B \in \ZZ$, set
\begin{align*}
  I^{(2)}(B)
\text{,}
&=
  \big\{ (n, r, m) \in \QQ^3 \,:\, 0 \le m, n < B,\, 4 n m - r^2 \ge 0 \big\}
\\[4pt]
  I^\rmJ(m; B)
&=
  \big\{ (n, r) \in \QQ^2 \,:\, 0 \le n < B,\, 4 n m - r^2 \ge 0 \big\}
\text{.}
\end{align*}
By abuse of notation, we denote the Fourier expansion map always by the same symbol~$\cF\cE_B$.  Given any representation $\rho$ of $\tSp{2}(\ZZ)$ or $\Gamma^\rmJ$ with representation space $V_\rho$, define
\begin{alignat*}{2}
  \cF\cE_B &:\, \rmM^{(1)}_{k}(\rho) \rightarrow V_\rho^B
\text{,}
&\;\;
  \Phi &\longmapsto \big( c(\Phi; n) \big)_{0 \le n < B}
\text{;}
\\[2pt]
  \cF\cE_B &:\, \rmM^{(2)}_{k, l}(\rho) \rightarrow (V_{\sigma_l} \otimes V_\rho)^{I^{(2)}(B)}
\text{,}
&\;\;
  \Phi &\longmapsto \big( c(\Phi; n, r, m) \big)_{(n, r, m)}
\text{;}
\\[2pt]
  \cF\cE_B &:\, \rmJ_{k, m}(\sigma_l \otimes \rho) \rightarrow (V_{\sigma_l} \otimes V_\rho)^{I^\rmJ(m; B)}
\text{,}
&\;\;
  \phi &\longmapsto \big( c(\phi; n, r) \big)_{(n, r)}
\text{;}
\\[2pt]
  \cF\cE_B &:\, \rmF\rmM^{(2)}_{k, l}(\rho) \rightarrow \bigoplus_{0 \le m < B} (V_{\sigma_l} \otimes V_\rho)^{I^\rmJ(m; B)}
\text{,}
&\;\;
  (\phi_m)_m &\longmapsto \Big( \big( c(\phi_m; n, r) \big)_{(n, r)} \Big)_{0 \le m < B}
\text{.}
\end{alignat*}
There are truncation maps from $V^{I^\rmJ(m; B')}$ to $V^{I^\rmJ(m; B)}$ if $B' \ge B$.  When comparing Fourier expansions that have different ``precision'', we will freely apply them.

\begin{proposition}
\label{prop:elliptic-modular-forms-fourier-expansion-precision}
Given $0 \le k \in \frac{1}{2} \ZZ$ and a representation $\rho$ of $\tSp{1}(\ZZ)$, the map $\cF\cE_B :\, \rmM^{(1)}_k(\rho) \rightarrow V_\rho^B$ is injective for $B > \frac{k}{12} + 1$.
\end{proposition}
\begin{proof}
Suppose that there is a modular form $\phi \in \rmM^{(1)}_k(\rho)$ whose first $l := \big\lfloor \frac{k}{12} \big\rfloor + 1$ Fourier coefficients vanish.  Then $\Delta^{-l} \phi$, where $\Delta$ is the discriminant form, is a non-zero, holomorphic modular form of negative weight, which cannot be.
\end{proof}

\begin{proposition}
\label{prop:jacobi-forms-fourier-expansion-precision}
Given $0 \le k \in \frac{1}{2}\ZZ$, $0 < m \in \QQ$, and a representation $\rho$ of $\tGamma^\rmJ$ with representation space $V_\rho$, the map $\cF\cE_B :\, \rmJ_{k,m}(\sigma_l \otimes \rho) \rightarrow \big( V_{\sigma_l} \otimes V_\rho \big)^{I^\rmJ(m; B)}$ is injective for $B > \frac{k + l + 2 \lfloor m \rfloor}{12} + 1$.
\end{proposition}
\begin{proof}
The bijection
\begin{gather*}
  \rmJ_{k, m}(\sigma_l \otimes \rho)
\cong
  \bigoplus_{i = 0}^l \rmJ_{k + i, m}(\rho)
\end{gather*}
is induced by differential operators with constant coefficients.  Hence it is compatible with taking Fourier expansions.  It is thus sufficient to treat the case $l = 0$.

We have
\begin{gather*}
  \cF\cE_B\big( \rmJ_{k,m}(\rho) \big)
\hookrightarrow
  \begin{cases}
    \bigoplus_{\nu = 0}^{\lfloor m \rfloor} \cF\cE_B\big( \rmM^{(1)}_{k + 2 \nu}(\rho) \big)
  \text{,}
  &
    \text{if $(-1)^k \rho(-I_2)$ is trivial;}
  \\
    \bigoplus_{\nu = 0}^{\lfloor m \rfloor - 1} \cF\cE_B\big( \rmM^{(1)}_{k + 2 \nu + 1}(\rho) \big)
  \text{,}
  &
    \text{otherwise.}
  \end{cases}
\end{gather*}
Elements of $\rmM^{(1)}_{k + i}(\rho)$, $0 \le i \le 2 \lfloor m \rfloor$  are uniquely determined by their first $\frac{k + 2 \lfloor m \rfloor}{12} + 1$ Fourier coefficients.  This proves the proposition.
\end{proof}

\begin{proposition}
\label{prop:formal-fourier-jacobi-determination-of-coefficients--truncated-version}
Let $k \in \frac{1}{2} \ZZ$, $l \in \ZZ$, and $0 < m \in \QQ$.  Further, let $\rho$ be a representation of $\tGamma^\rmJ$ with representation space $V_\rho$.  Given $0 < B \in \ZZ$, fix families
\begin{gather*}
  (\phi_m)_{0 \le m < B},\,
  (\psi_m)_{0 \le m < B}
\in
  \bigoplus_{0 \le m < B} \rmJ_{k, m}(\sigma_l \otimes \rho)
\text{.}
\end{gather*}
Suppose they satisfy $\cF\cE_{\frac{k + l}{10} + 1}( \phi_m ) = \cF\cE_{\frac{k + l}{10} + 1}( \psi_m )$ for all $0 \le m \le \frac{k + l}{10}$ and
\begin{align*}
  c(\phi_m; n, r)
&=
  (\det{}^k \otimes \sigma_l)^{-1}(S)\, \rho^{-1}(\rot(S))\, c(\phi_n; m, r)
\text{,}
\\
  c(\psi_m; n, r)
&=
  (\det{}^k \otimes \sigma_l)^{-1}(S)\, \rho^{-1}(\rot(S))\, c(\psi_n; m, r)
\text{,}
\end{align*}
for all $0 < n, m < B$ and all $r \in \ZZ$.  Then we have $\phi_m = \psi_m$ for all $0 \le m < B$.
\end{proposition}
\begin{proof}
By Proposition~\ref{prop:jacobi-forms-fourier-expansion-precision}, $\phi_m$ and $\psi_m$ are determined by $c(\phi_m; n, r)$ and $c(\psi_m; n, r)$ with $n \le \frac{k + l + 2 \lfloor m \rfloor}{12}$.  For $m \le \frac{k + l}{10}$, we have $\frac{k + l + 2m}{12} < \frac{k + l}{10} + 1$, and hence $\phi_m = \psi_m$.

The rest of the proof is analog to the proof of Proposition~\ref{prop:formal-fourier-jacobi-determination-of-coefficients}.  It suffices to show that $\phi_m$, $m > \frac{k + l}{10}$ is uniquely determined by all $\phi_{m'}$, $m' < m$.  Employing Proposition~\ref{prop:jacobi-forms-fourier-expansion-precision}, this follows from the inequality $\frac{k + l + 2 m}{12} < m$, which holds exactly if $m > \frac{k + l}{10}$.
\end{proof}

Given $k \in \frac{1}{2} \ZZ$, $l \in \ZZ$, a unitary representation $\rho$ of $\tSp{2}(\ZZ)$, and $0 < B \in \ZZ$, define
\begin{align*}
  \rmF\rmM^{(2)}_{k, l}(\rho)_B
=
  \Big\{
&
    (\phi_m)_{0 \le m < B} \in \bigoplus_{0 \le m < B} \cF\cE_B \big( \rmJ_{k, m}(\sigma_l \otimes \rho) \big)
      \,:\,
\\[2pt]
&\quad
    c(\phi_m; n, r) = (\det{}^k \otimes \sigma_l)^{-1}(S) \, \rho^{-1}(\rot(S))\, c(\phi_n; m, r)
\\
&\quad
    \text{for all $m, n < B$ and all $r \in \ZZ$}
  \Big\}
\text{.}
\end{align*}

\begin{proposition}
\label{prop:inclusion-of-spaces-of-truncated-formal-fourier-jacobi-expansions}
Given $k \in \frac{1}{2} \in \ZZ$, $l \in \ZZ$, a unitary representation $\rho$ of $\tSp{2}(\ZZ)$ with representation space $V_\rho$, and $\frac{k + l}{10} < B \le B' \in \ZZ$ the following inclusion holds.
\begin{gather*}
  \rmF\rmM^{(2)}_{k, l}(\rho)_{B'}
\subseteq
  \rmF\rmM^{(2)}_{k, l}(\rho)_B
\text{,}
\end{gather*}
where we implicitly truncate elements of the space on the right hand side.
\end{proposition}
\begin{proof}
This follows from Proposition~\ref{prop:formal-fourier-jacobi-determination-of-coefficients--truncated-version}:  Given the first $\frac{k + l}{10}$ Fourier Jacobi coefficient of any element in $\rmF\rmM^{(2)}_{k, l}(\rho)_{B}$ or $\rmF\rmM^{(2)}_{k, l}(\rho)_{B'}$, all others are uniquely defined.
\end{proof}

\begin{algorithm}
\label{alg:compute-via-fj-expansion:compute-jacobi}
Let $0 \le k \in \frac{1}{2}\ZZ$, $0 \le l \in \ZZ$, and let $\rho$ be a unitary representation of $\tSp{2}(\ZZ)$.  Given $\frac{k + l}{10} < B \in \ZZ$, the following algorithm computes the space $\cF\cE_B\big( \rmM^{(2)}_{k,l}(\rho) \big)$.
\begin{enumerate}[(1)]
\item \label{it:alg:compute-via-fj-expansion:compute-jacobi}
  Compute $\cF\cE_B\big( \rmJ_{k, m}(\sigma_l \otimes \rho) \big)$ for all $0 \le m < B$.
\item \label{it:alg:compute-via-fj-expansion:solve-system-of-linear-equations}
  Compute $\rmF\rmM^{(2)}_{k, l}(\rho)_B$.
\item \label{it:alg:compute-via-fj-expansion:compare-dimensions}
  If $\dim \rmF\rmM^{(2)}_{k, l}(\rho)_B = \dim \rmM^{(2)}_{k, l}(\rho)$, then we are done.  Otherwise, increase $B$ and go back to Step~\ref{it:alg:compute-via-fj-expansion:compute-jacobi}.
\end{enumerate}
After processing these steps, there is a one-to-one correspondence of elements $\Phi \in \cF\cE_B\big( \rmM^{(2)}_{k,l}(\rho) \big)$ and elements $(\phi_m) \in \rmF\rmM^{(2)}_{k, l}(\rho)_{B}$ via $c(\Phi; n, r, m) = c(\phi_m; n, r)$.
\end{algorithm}
\begin{remarks}
All steps except for the last on can be implemented based on the current state of knowledge.
\begin{enumerate}[(1)]
\item Step~\ref{it:alg:compute-via-fj-expansion:compute-jacobi} requires use of results in~\cite{IPY12} and~\cite{Ra12}.  In the former paper an explicit bijection of $\rmJ_{k,m}(\sigma_l \otimes \rho)$ and $\bigoplus_{i = 0}^l \rmJ_{k + i, m}(\rho)$ was given.  In order to compute $\cF\cE_B\big( \rmJ_{k, m}(\rho) \big)$ using the latter work, note that $\rmJ_{k, m}(\rho) \cong \rmM^{(1)}_{k - \frac{1}{2}}( {\check \rho}_m \otimes \rho )$ (${\check \rho}_m$ is the dual of the Weil representation for the lattice $(2 m)$ -- see~\cite{Ra12}) via theta decomposition.

\item Step~\ref{it:alg:compute-via-fj-expansion:solve-system-of-linear-equations} can be done by means of basic linear algebra.

\item Step~\ref{it:alg:compute-via-fj-expansion:compare-dimensions} requires computations of $\dim \rmM^{(2)}_{k, l}(\rho)$.  If $\rho$ is trivial and $k \ge 4$ ($l = 0$) or $k \ge 5$ ($l > 0$), this has been done by Tsushima~\cite{Ts83}.  To the author's knowledge, the case of non-trivial $\rho$ has not yet been treated.
\end{enumerate}
\end{remarks}
\begin{proof}
The algorithm terminates, because we have
\begin{gather*}
  \rmF\rmM^{(2)}_{k, l}(\rho)
=
  \bigcap_{0 < B \in \ZZ} \rmF\rmM^{(2)}_{k, l}(\rho)_B
\text{,}
\end{gather*}
and $\dim \rmF\rmM^{(2)}_{k, l}(\rho)_B$ is monotonically decreasing for sufficiently large $B$ by Proposition~\ref{prop:inclusion-of-spaces-of-truncated-formal-fourier-jacobi-expansions}.

Correctness of Algorithm~\ref{alg:compute-via-fj-expansion:compute-jacobi} follows, because we have
\begin{gather*}
  \cF\cE_B\big( \rmM^{(2)}_{k, l}(\rho) \big)
\subseteq
  \rmF\rmM^{(2)}_{k, l}(\rho)_B
\end{gather*}
for all $0 < B \in \ZZ$.  If the dimension check in Step~\ref{it:alg:compute-via-fj-expansion:compare-dimensions} succeeds, then equality holds.
\qedhere


\end{proof}

\bibliographystyle{amsalpha}
\bibliography{bibliography}

\end{document}